\documentclass[10pt]{NSP1}
\usepackage{url,floatflt}
\usepackage{helvet,times}
\usepackage{psfig,graphics}
\usepackage{mathptmx,amsmath,amssymb,bm}
\usepackage{float}
\usepackage[bf,hypcap]{caption}

\emergencystretch=\maxdimen
\def\no{\noindent}

\def\firstpage{1}
\def\thevol{1}
\def\thenumber{1}
\def\theyear{201?}

\def\Lip{\operatorname{Lip}}
\def\C{C}

\def\Linf{L^\infty}
\def\diam{\operatorname{diam}}
\newcommand{\argmax}[1]{\underset{#1}{\operatorname{argmax} \,}}
\newcommand{\argmin}[1]{\underset{#1}{\operatorname{argmin} \,}}
\newcommand{\esssup}[1]{\underset{#1}{\operatorname{ess\,sup} \,}}

\begin{document}

\titlefigurecaption{{\large \bf \rm Progress in Fractional Differentiation and Applications}\\ {\it\small An International Journal}}

\title{Approximation of Fractional Order \\ Conflict-Controlled Systems}

\author{Mikhail Gomoyunov{$^{1,2}$}}
\institute{Krasovskii Institute of Mathematics and Mechanics, Ural Branch of Russian Academy of Sciences, Ekaterinburg 620990, Russia\and
Ural Federal University, Ekaterinburg 620002, Russia}

\titlerunning{Approximation of Fractional Order Conflict-Controlled Systems}
\authorrunning{M. Gomoyunov}

\mail{m.i.gomoyunov@gmail.com}

\received{26 Apr. 2018}
\revised{???}
\accepted{???}
\published{???}

\abstracttext{We consider a conflict-controlled dynamical system described by a nonlinear ordinary fractional differential equation with the Caputo derivative of an order $\alpha \in (0, 1).$ Basing on the finite-difference Gr\"{u}nwald-Letnikov formulas, we propose an approximation of the considered system by a system described by a functional-differential equation of a retarded type. A mutual aiming procedure between the initial conflict-controlled system and the approximating system is given that guarantees the desired proximity between their motions. This procedure allows to apply, via the approximating system, the results obtained for functional-differential systems for solving control problems in fractional order systems. Examples are considered, results of numerical simulations are presented.}

\keywords{Fractional differential equation, approximation, fractional order difference, control problem, disturbances.\\[2mm]
}

\maketitle

\section{\; Introduction}

The paper deals with nonlinear ordinary fractional differential equations with the Caputo derivative of an order $\alpha \in (0, 1).$ The basics of the theory and numerical methods for such equations can be found, for example, in \cite{SGSamko_AAKilbas_OIMarichev_1993,IPodlubny_1999,AAKilbas_HMSrivastava_JJTrujillo_2006,KDiethelm_2010,CLi_FZeng_2015}, where also some of their applications are presented.

In the first part of the paper, we propose an approximation of a fractional differential equation by a functional-differential equation of a retarded type (see, e.g., \cite{JKHale_SMVLunel_1993}). The approximation is based on the finite-difference Gr\"{u}nwald-Letnikov (G.-L.) formulas for calculation of fractional derivatives (see, e.g., \cite[p.~386]{SGSamko_AAKilbas_OIMarichev_1993}), and it is arranged in such a way that the solution of the considered functional-differential equation approximates not the solution of the initial equation, but its Riemann-Liouville (R.-L.) fractional integral of the order $1 - \alpha.$ The proof of this fact relies on the uniform Lipschitz continuity of solutions of the approximating equation and the estimate of a fractional derivative of a quadratic Lyapunov function \cite{MIGomoyunov_2017}.

In the second part of the paper, we consider a conflict-controlled dynamical system which motion is described by a fractional differential equation. To apply the proposed approximation for this case, following \cite{NNKrasovskii_ANKotelnikova_2012} (see also \cite{NYuLukoyanov_ARPlaksin_2015,NYuLukoyanov_ARPlaksin_2016}), a mutual aiming procedure between the initial conflict-controlled system and the corresponding approximating functional-differential system is elaborated. It is based on the extremal shift rule (see, e.g., \cite[\S\S~2.4, 8.2]{NNKrasovskii_AISubbotin_1988} and also \cite{MIGomoyunov_2017}) and guarantees the desired proximity between systems' motions. This procedure allows to apply, via the approximating system, the results obtained for functional-differential systems (see, e.g., \cite{YuSOsipov_1971,NYuLukoyanov_2001,NYuLukoyanov_2003}) for solving control problems in fractional order systems. It should be noted also that, according to \cite{NYuLukoyanov_ARPlaksin_2015} (see also references therein), the approximating system can be further approximated by a high-dimensional system of the usual ordinary differential equations.

The paper is organized as follows. In Sect.~\ref{sec_preliminaries}, we introduce the notations, recall the definitions of fractional order integrals, derivatives and differences and give some of their properties. In Sect.~\ref{sec_DE}, we consider a Cauchy problem for an ordinary differential equation with the Caputo fractional derivative. In Sect.~\ref{sec_approximation_DE}, we derive an approximating functional-differential equation, establish its properties and prove the corresponding approximation theorem. The obtained results are illustrated by an example in Sect.~\ref{sec_Example_1}. In Sect.~\ref{sec_DS}, we consider a conflict-controlled fractional order dynamical system and introduce the approximating dynamical system. A mutual aiming procedure that ensures the desired proximity between motions of these systems is proposed in Sect.~\ref{sec_Procedure}, the corresponding approximation theorem is proved. In Sect.~\ref{sec_Example_2}, an illustrating example is considered. Concluding remarks are given in Sect.~\ref{sec_Conclusion}.

\section{\; Preliminaries} \label{sec_preliminaries}

\subsection{\; Notations} \label{subsec_notations}

Let numbers $\alpha \in (0, 1),$ $T > 0$ and $n \in \mathbb{N}$ be fixed throughout the paper. Let $\mathbb{R}^n$ be the $n$-dimen\-sional Euclidian space with the scalar product $\langle \cdot, \cdot \rangle$ and the norm $\|\cdot\|.$ By $B(r) \subset \mathbb{R}^n$ for $r \geq 0,$ we denote the closed ball with the center in the origin and the radius $r.$ The segment $[0, T] \subset \mathbb{R}$ is assumed to be endowed with the Lebesgue measure. By $\Linf = \Linf([0, T], \mathbb{R}^n),$ we denote the space of (classes of equivalence of) essentially bounded measurable functions $x: [0, T] \rightarrow \mathbb{R}^n$ with the norm
\begin{equation*}
    \|x(\cdot)\|_\infty = \esssup{t \in [0, T]} \|x(t)\|.
\end{equation*}
Let $\C = \C([0, T], \mathbb{R}^n)$ be the space of continuous functions $x: [0, T] \rightarrow \mathbb{R}^n$ with the norm $\|\cdot\|_\infty.$ By $\Lip^0 = \Lip^0([0, T], \mathbb{R}^n)$ we denote the set of Lipschitz continuous functions $x(\cdot) \in \C$ such that $x(0) = 0.$ By $\Lip_L^0 = \Lip_L^0([0, T], \mathbb{R}^n)$ for $L \geq 0,$ we denote the set of functions $x(\cdot) \in \Lip^0$ which are Lipschitz continuous with the constant $L.$

\subsection{\; Fractional Order Integral and Derivatives} \label{subsec_RL_integral}
Let us recall the definitions of the R.-L. fractional integral, R.-L. and Caputo fractional derivatives and give some of their properties, which are used in the paper below.

\begin{definition}[see {\cite[Definition~2.1]{SGSamko_AAKilbas_OIMarichev_1993}}] \
    For a function $\varphi: [0, T] \rightarrow~\mathbb{R}^n,$ the {\rm(}left-sided{\rm)} R.-L. fractional integral of the order $\alpha$ is defined by
    \begin{equation*}
        (I^\alpha \varphi) (t) = \frac{1}{\Gamma(\alpha)} \int_{0}^{t} \frac{\varphi(\tau)}{(t-\tau)^{1-\alpha}} d\tau, \quad t \in [0, T],
    \end{equation*}
    where $\Gamma(\cdot)$ is the Euler gamma-function {\rm(}see, e.g., {\rm\cite[(1.54)]{SGSamko_AAKilbas_OIMarichev_1993})}.
\end{definition}

According to \cite[Proposition~1.1]{MIGomoyunov_2017}, the following proposition holds.
\begin{proposition} \label{prop_integral_properties} \
    For any $\varphi(\cdot) \in \Linf,$ the value $(I^\alpha \varphi)(t)$ is well defined for $t \in [0, T],$ and $(I^\alpha \varphi)(0) = 0.$ Moreover, there exists $H_\alpha > 0$ such that, for any $\varphi(\cdot) \in \Linf,$ the inequality
    \begin{equation*}
        \|(I^\alpha \varphi) (t) - (I^\alpha \varphi) (\tau)\|
        \leq H_\alpha \|\varphi(\cdot)\|_\infty |t - \tau|^\alpha, \quad t, \tau \in [0, T],
    \end{equation*}
    is valid. In particular, $(I^\alpha \varphi)(\cdot) \in \C$ for any $\varphi(\cdot) \in \Linf.$
\end{proposition}

\begin{proposition} \label{Prop_integral_properties_2} \
    Let a function $\varphi(\cdot) \in \C([0, T], \mathbb{R})$ be non-decreasing and non-negative. Then the function $(I^\alpha \varphi)(\cdot)$ is non-decreasing.
\end{proposition}
\begin{proof} \
    Let $t, \tau \in [0, T]$ and $t > \tau.$ We have
    \begin{equation*}
        (I^\alpha \varphi) (t) - (I^\alpha \varphi) (\tau)
        = \frac{1}{\Gamma(\alpha)} \int_{\tau}^{t} \frac{\varphi(\xi)}{(t - \xi)^{1 - \alpha}} d \xi
        + \frac{1}{\Gamma(\alpha)} \int_{0}^{\tau} \varphi(\xi) \big( (t - \xi)^{\alpha - 1} - (\tau - \xi)^{\alpha - 1} \big) d \xi.
    \end{equation*}
    Since $\varphi(\cdot)$ is non-decreasing, for the first term, we obtain
    \begin{equation*}
            \frac{1}{\Gamma(\alpha)} \int_{\tau}^{t} \frac{\varphi(\xi)}{(t - \xi)^{1 - \alpha}} d \xi
            \geq \frac{\varphi(\tau)}{\Gamma(\alpha)} \int_{\tau}^{t} (t - \xi)^{\alpha - 1} d \xi
            = \frac{\varphi(\tau) (t - \tau)^\alpha}{\Gamma(\alpha + 1)},
    \end{equation*}
    and, for the second term, we derive
    \begin{equation*}
        \frac{1}{\Gamma(\alpha)} \int_{0}^{\tau} \varphi(\xi) \big( (t - \xi)^{\alpha - 1} - (\tau - \xi)^{\alpha - 1} \big) d \xi
        \geq \frac{\varphi(\tau)}{\Gamma(\alpha)} \int_{0}^{\tau} \big( (t - \xi)^{\alpha - 1} - (\tau - \xi)^{\alpha - 1} \big) d \xi
        = \frac{\varphi(\tau) ( t^\alpha - \tau^\alpha - (t - \tau)^\alpha)}{\Gamma(\alpha + 1)}.
    \end{equation*}
    Therefore, taking into account that $\varphi(\cdot)$ is non-negative, we deduce
    \begin{equation*}
            (I^\alpha \varphi) (t) - (I^\alpha \varphi) (\tau)
            \geq \varphi(\tau) (t^\alpha - \tau^\alpha) / \Gamma(\alpha + 1) \geq 0.
    \end{equation*}
    The proposition is proved.
\end{proof}

\begin{definition}[see {\cite[Definition~2.2]{SGSamko_AAKilbas_OIMarichev_1993}}] \
    For a function $x: [0, T] \rightarrow~\mathbb{R}^n,$ the {\rm(}left-sided{\rm)} R.-L. fractional derivative of the order $\alpha$ is defined by
    \begin{equation*}
        (D^\alpha x) (t)
        = \frac{1}{\Gamma(1 - \alpha)} \frac{d}{dt} \int_{0}^{t} \frac{x(\tau)}{(t - \tau)^{\alpha}} d\tau, \quad t \in [0, T].
    \end{equation*}
\end{definition}

Let us denote by $I^\alpha (\Linf)$ the set of functions $x: [0, T] \rightarrow \mathbb{R}^n$ represented by the R.-L. fractional integral of the order $\alpha$ of a function $\varphi(\cdot) \in \Linf$: $x(t) = (I^\alpha \varphi)(t),$ $t \in [0, T].$

The next two propositions follows from \cite[Proposition~1.2]{MIGomoyunov_2017}.
\begin{proposition} \label{prop_derivative_properties} \
    For any $x(\cdot) \in I^\alpha(\Linf),$ the value $(D^\alpha x)(t)$ is well defined for almost every $t \in [0, T],$ and the inclusion $(D^\alpha x)(\cdot) \in \Linf$ is valid. Moreover, for any $\varphi(\cdot) \in \Linf,$ the equality $\varphi(t) = (D^\alpha x)(t)$ holds for almost every $t \in [0, T]$ if and only if $(I^\alpha \varphi)(t) = x(t),$ $t \in [0, T].$
\end{proposition}

\begin{proposition} \label{prop_derivative_properties_2} \
    For any $x(\cdot) \in \Lip^0,$ the inclusion $x(\cdot) \in I^\alpha(\Linf)$ is valid, and the value $(D^\alpha x) (t)$ is well defined for $t \in [0, T].$ Moreover, for any $\varphi(\cdot) \in \Linf$ such that $\varphi(t) = \dot{x}(t)$ for almost every $t \in [0, T],$ where $\dot{x}(t) = d x(t)/dt,$ the equality $(D^\alpha x)(t) = (I^{1 - \alpha} \varphi)(t),$ $t \in [0, T],$ holds. In particular, we have $(D^\alpha x)(\cdot) \in I^{1 - \alpha}(\Linf),$ $(D^\alpha x)(0) = 0$ and $(D^{1 - \alpha}(D^\alpha x))(t) = \dot{x}(t)$ for almost every $t \in [0, T].$
\end{proposition}

\begin{proposition} \label{prop_RL_derivative_Lipshitz} \
    For any $L > 0,$ there exist $K > 0$ and $M > 0$ such that, for any $x(\cdot) \in \Lip_L^0,$ the inequalities below hold:
    \begin{equation*} \label{prop_RL_derivative_Lipshitz_main}
        \| (D^\alpha x) (t) - (D^\alpha x) (\tau)\| \leq K |t - \tau|^{1 - \alpha}, \quad
        \| (D^\alpha x) (t) \| \leq M, \quad t, \tau \in [0, T].
    \end{equation*}
\end{proposition}
\begin{proof} \
    Let $L > 0$ be fixed. Taking $H_{1-\alpha}$ from Proposition~\ref{prop_integral_properties}, we define $K = H_{1-\alpha} L,$ $M = K T^{1 - \alpha}.$ Let $x(\cdot) \in \Lip_L^0,$ and $\varphi(\cdot) \in \Linf$ be such that $\varphi(t) = \dot{x}(t)$ for almost every $t \in [0, T].$ Then $\|\varphi(\cdot)\|_\infty \leq L$ and, due to Proposition~\ref{prop_derivative_properties_2}, we have $(D^\alpha x)(t) = (I^{1 - \alpha} \varphi)(t),$ $t \in [0, T].$ Therefore, by the choice of $H_{1 - \alpha},$ we obtain
    \begin{equation*}
        \|(D^\alpha x)(t) - (D^\alpha x)(\tau)\| \leq H_{1 - \alpha} L |t - \tau|^{1 - \alpha} = K |t - \tau|^{1 - \alpha},
        \quad t, \tau \in [0, T].
    \end{equation*}
    Further, since $(D^\alpha x)(0) = 0$ according to Proposition~\ref{prop_derivative_properties_2}, we derive
    \begin{equation*}
        \|(D^\alpha x)(t)\| \leq K t^{1 - \alpha} \leq M, \quad t \in [0, T].
    \end{equation*}
    The proposition is proved.
\end{proof}

\begin{definition}[see {\cite[(2.4.1)]{AAKilbas_HMSrivastava_JJTrujillo_2006}}] \
    For a function $x: [0, T] \rightarrow \mathbb{R}^n,$ the {\rm(}left-sided{\rm)} Caputo fractional derivative of the order $\alpha$ is defined by
    \begin{equation} \label{def_Caputo}
        ({}^C D^\alpha x) (t)
        = \frac{1}{\Gamma(1 - \alpha)} \frac{d}{dt} \int_{0}^{t} \frac{x(\tau) - x(0)}{(t - \tau)^{\alpha}} d\tau, \quad t \in [0, T].
    \end{equation}
\end{definition}

Note that, for a function $x:[0, T] \rightarrow \mathbb{R}^n,$ if $x(0) = 0,$ then the Caputo and the R.-L. fractional derivatives coincide.

\subsection{\; Fractional Order Differences} \label{subsec_differences}
Let us recall the notion of the fractional order difference, on which the definition of the G.-L. fractional order derivative is based (see, e.g., \cite[\S20.4]{SGSamko_AAKilbas_OIMarichev_1993}), and prove some auxiliary statements, which are used in the paper below.

\begin{definition}[see {\cite[p.~385]{SGSamko_AAKilbas_OIMarichev_1993}}] \
    For a function $x: [0, T] \rightarrow~\mathbb{R}^n,$ the {\rm(}left-sided{\rm)} fractional difference of the order $\alpha$ with a step size $h > 0$ is defined by
    \begin{equation*} \label{Delta_h^alpha}
        (\Delta_h^\alpha x) (t) = \sum_{j = 0}^{[t/h]} (-1)^j \binom{\alpha}{j} x(t - j h), \quad t \in [0, T],
    \end{equation*}
    where the symbol $[\tau]$ means the integer part of $\tau \geq 0,$ and $\binom{\alpha}{j}$ are the binomial coefficients.
\end{definition}

Let us consider the function
\begin{equation} \label{p_alpha}
    p_\alpha (\tau) = \frac{1}{\Gamma(\alpha)} \sum_{0 \leq j < \tau} (- 1)^j \binom{\alpha}{j} \frac{1}{(\tau - j)^{1 - \alpha}},
    \quad \tau > 0.
\end{equation}
Note that the function $p_\alpha(\cdot)$ is measurable and, according to \cite[Lemma~20.1]{SGSamko_AAKilbas_OIMarichev_1993} (see also \cite[Lemma~2]{UWestphal_1974}), the following relations hold:
\begin{equation} \label{p_alpha_properties}
    \|p_\alpha(\cdot)\|_1 = \int_0^{\infty} |p_\alpha (\tau)| d \tau < \infty, \quad
    \int_0^{\infty} p_\alpha (\tau) d \tau = 1.
\end{equation}
Connection between the function $p_\alpha(\cdot)$ and the fractional difference of the order $\alpha$ is given in the following proposition (see also \cite[(20.30)]{SGSamko_AAKilbas_OIMarichev_1993}).
\begin{proposition} \label{prop_p_alpha} \
    Let $x(\cdot) \in I^\alpha (\Linf),$ and $\varphi(\cdot) \in \Linf$ be such that $(D^\alpha x)(t) = \varphi(t)$ for almost every $t \in [0, T].$ Then, for any $h > 0,$ we have
    \begin{equation} \label{prop_p_alpha_main}
        h^{- \alpha} (\Delta_h^\alpha x)(t)
        = \int_{0}^{t/h} \varphi(t - \tau h) p_\alpha(\tau) d \tau, \quad t \in [0, T].
    \end{equation}
\end{proposition}
\begin{proof} \
    Let $t \in [0, T].$ Due to Proposition~\ref{prop_derivative_properties}, we derive
    \begin{equation*}
        (\Delta_h^\alpha x) (t)
        = \sum_{j = 0}^{[t/h]} (-1)^j \binom{\alpha}{j} \frac{1}{\Gamma(\alpha)} \int_{0}^{t - j h}
        \frac{\varphi(\xi)}{(t - j h - \xi)^{1 - \alpha}} d \xi
        = \int_0^t \varphi(\xi) \frac{1}{\Gamma(\alpha)} \sum_{j = 0}^{[t/h]} (-1)^j \binom{\alpha}{j}
        \frac{\chi(t - j h - \xi)}{(t - j h - \xi)^{1 - \alpha}} d \xi,
    \end{equation*}
    where we denote $\chi(\xi) = 0$ for $\xi \leq 0,$ and $\chi(\xi) = 1$ for $\xi > 0.$ Changing the variable $\xi = t - h \tau,$ we obtain
    \begin{equation*}
        \begin{array}{c}
            (\Delta_h^\alpha x) (t)
            \displaystyle = h^\alpha \int_{0}^{t/h} \varphi(t - \tau h) \frac{1}{\Gamma(\alpha)} \sum_{j = 0}^{[t/h]} (-1)^j \binom{\alpha}{j}
            \frac{\chi(\tau - j)}{(\tau - j)^{1 - \alpha}} d \tau \\[1em]
            \displaystyle = h^\alpha \int_{0}^{t/h} \varphi(t - \tau h) \frac{1}{\Gamma(\alpha)} \sum_{0 \leq j < \tau} (-1)^j \binom{\alpha}{j}
            \frac{1}{(\tau - j)^{1 - \alpha}} d \tau,
        \end{array}
    \end{equation*}
    wherefrom, according to (\ref{p_alpha}), we derive (\ref{prop_p_alpha_main}).
\end{proof}

The next result is an analog of Proposition~\ref{prop_RL_derivative_Lipshitz}.
\begin{proposition} \label{prop_GL_derivative_Lipshitz} \
    For any $L > 0,$ there exist $\overline{K} > 0$ and $\overline{M} > 0$ such that, for any $h > 0$ and any $x(\cdot) \in \Lip_L^0,$ the inequalities below are valid:
    \begin{equation} \label{prop_GLD_main}
        \|h^{- \alpha} (\Delta_h^\alpha x)(t) - h^{- \alpha} (\Delta_h^\alpha x)(\tau)\| \leq \overline{K} |t - \tau|^{1 - \alpha}, \quad
        \|h^{- \alpha} (\Delta_h^\alpha x)(t)\| \leq \overline{M}, \quad t, \tau \in [0, T].
    \end{equation}
\end{proposition}
\begin{proof} \
    Let $L > 0$ be fixed. Let us choose $K > 0$ by Proposition~\ref{prop_RL_derivative_Lipshitz} and define $\overline{K} = 2 K \|p_\alpha(\cdot)\|_1,$ $\overline{M} = \overline{K} T^{1 - \alpha}.$ Let $h > 0,$ $x(\cdot) \in \Lip_L^0$ and $\varphi(t) = (D^\alpha x)(t),$ $t \in [0, T].$ Let $t, \tau \in [0, T]$ and, for simplicity, $t > \tau.$ Due to Proposition~\ref{prop_p_alpha}, we have
    \begin{equation*}
        \begin{array}{c}
            \displaystyle h^{-\alpha} \| (\Delta_h^\alpha x)(t) - (\Delta_h^\alpha x)(\tau) \|
            = \Big\|\int_{0}^{t/h} \varphi(t - \xi h) p_\alpha(\xi) d \xi
            - \int_{0}^{\tau/h} \varphi(\tau - \xi h) p_\alpha(\xi) d \xi \Big\| \\[1em]
            \displaystyle \leq \Big\|\int_{0}^{\tau/h} \big( \varphi(t - \xi h) - \varphi(\tau - \xi h) \big) p_\alpha(\xi) d \xi \Big\|
            + \Big\|\int_{\tau/h}^{t/h} \varphi(t - \xi h) p_\alpha(\xi) d\xi\Big\|.
        \end{array}
    \end{equation*}
    By the choice of $K,$ for the first term, we obtain
    \begin{equation*}
        \begin{array}{c}
            \displaystyle \Big\|\int_{0}^{\tau/h} \big( \varphi(t - \xi h) - \varphi(\tau - \xi h) \big) p_\alpha(\xi) d \xi \Big\|
            \leq \int_{0}^{\tau/h} \big\| \varphi(t - \xi h) - \varphi(\tau - \xi h) \big\| |p_\alpha(\xi)| d \xi \\[1em]
            \displaystyle \leq K (t - \tau)^{1 - \alpha} \int_{0}^{\tau/h} |p_\alpha(\xi)| d \xi
            \leq K (t - \tau)^{1 - \alpha} \|p_\alpha(\cdot)\|_1,
        \end{array}
    \end{equation*}
    and for the second term, since $\varphi(0) = 0$ by Proposition~\ref{prop_derivative_properties_2}, we derive
    \begin{equation*}
        \Big\|\int_{\tau/h}^{t/h} \varphi(t - \xi h) p_\alpha(\xi) d\xi\Big\|
        \leq \int_{\tau/h}^{t/h} K |t - \xi h|^{1 - \alpha} |p_\alpha(\xi)| d\xi
        \leq K (t - \tau)^{1 - \alpha} \int_{\tau/h}^{t/h} |p_\alpha(\xi)| d\xi
        \leq K (t - \tau)^{1 - \alpha} \|p_\alpha(\cdot)\|_1.
    \end{equation*}
    Thus, the first inequality in (\ref{prop_GLD_main}) is proved for the chosen $\overline{K}.$ The second inequality in (\ref{prop_GLD_main}) follows from the first one if we take into account that $(\Delta_h^\alpha x)(0) = x(0) = 0$ and the choice of $\overline{M}.$
\end{proof}

Due to relations (\ref{p_alpha_properties}), the following approximation property holds.
\begin{proposition} \label{prop_p_alpha_approximation} \
    Let $W \subset \C$ be a relatively compact set such that $\varphi(0) = 0$ for any $\varphi(\cdot) \in W.$ Then, for any $\varepsilon > 0,$ there exists $h_\ast > 0$ such that, for any $h \in (0, h_\ast]$ and any $\varphi(\cdot) \in W,$ the inequality below is valid:
    \begin{equation} \label{prop_p_alpha_approximation_main}
        \Big\| \int_{0}^{t/h} \varphi(t - \tau h) p_\alpha (\tau) d \tau - \varphi(t) \Big\| \leq \varepsilon, \quad t \in [0, T].
    \end{equation}
\end{proposition}
\begin{proof} \
    The proof follows the arguments from \cite[\S~2]{RADeVore_GGLorentz_1993}. According to Arcel\`{a}-Ascoli theorem (see, e.g., \cite[Ch.~I, \S~5, Theorem~4]{LVKantorovich_GPAkilov_1982}), one can choose $A > 0$ and $\delta > 0$ such that, for any $\varphi(\cdot) \in W,$ the inequality $\|\varphi(t)\| \leq A,$ $t \in [0, T],$ holds and the estimate $\|\varphi(t) - \varphi(\tau)\| \leq \varepsilon/(2 \|p_\alpha(\cdot)\|_1)$ is valid for any $t, \tau \in [0, T]$ such that $|t - \tau| \leq \delta.$ Basing on the first relation in (\ref{p_alpha_properties}), let us choose $N > 0$ from the condition
    \begin{equation*}
        \int_N^\infty |p_\alpha(\tau)| d \tau \leq \varepsilon / (4 A),
    \end{equation*}
    and put $h_\ast = \delta / N.$ Let $h \in (0, h_\ast]$ and $\varphi(\cdot) \in W.$ Taking into account that $\varphi(0) = 0,$ let us define $\varphi(t) = 0$ for $t < 0.$ Due to the second relation in (\ref{p_alpha_properties}), for any $t \in [0, T],$ we have
    \begin{equation*}
        \begin{array}{c}
            \displaystyle \Big\| \int_{0}^{t/h} \varphi(t - \tau h) p_\alpha (\tau) d \tau - \varphi(t) \Big\|
            = \Big\| \int_{0}^{\infty} (\varphi(t - \tau h) - \varphi(t)) p_\alpha (\tau) d \tau \Big\| \\[1em]
            \displaystyle \leq \int_{0}^{\delta / h} \|\varphi(t - \tau h) - \varphi(t)\| |p_\alpha (\tau)| d \tau
            + \int_{\delta/h}^{\infty} \|\varphi(t - \tau h) - \varphi(t)\| |p_\alpha (\tau)| d \tau.
        \end{array}
    \end{equation*}
    For the first term, by the choice of $\delta,$ we derive
    \begin{equation*}
        \int_{0}^{\delta / h} \|\varphi(t - \tau h) - \varphi(t)\| |p_\alpha (\tau)| d \tau
        \leq \frac{\varepsilon}{2 \|p_\alpha(\cdot)\|_1} \int_{0}^{\delta / h} |p_\alpha (\tau)| d \tau \leq \frac{\varepsilon}{2}.
    \end{equation*}
    For the second term, by the choice of $A,$ $h_\ast$ and $N,$ we deduce
    \begin{equation*}
        \int_{\delta/h}^{\infty} \|\varphi(t - \tau h) - \varphi(t)\| |p_\alpha (\tau)| d \tau
        \leq 2 A \int_{N}^{\infty} |p_\alpha (\tau)| d \tau \leq \frac{\varepsilon}{2}.
    \end{equation*}
    Thus, inequality (\ref{prop_p_alpha_approximation_main}) is proved.
\end{proof}

\begin{lemma} \label{lem_RL_GL_derivatives} \
    For any $L > 0$ and any $\varepsilon > 0,$ there exists $h_\ast >0$ such that, for any $h \in (0, h_\ast]$ and any $x(\cdot) \in \Lip_L^0,$ the following inequality holds:
    \begin{equation} \label{lem_RL_GL_derivatives_main}
        \| h^{-\alpha} (\Delta_h^\alpha x)(t) - (D^\alpha x) (t) \| \leq \varepsilon, \quad t \in [0, T].
    \end{equation}
\end{lemma}
\begin{proof} \
    Due to Proposition~\ref{prop_RL_derivative_Lipshitz}, by Arcel\`{a}-Ascoli theorem, the set
    \begin{equation*}
        W = \big\{ \varphi(\cdot) \in \C: \, \varphi(t) = (D^\alpha x)(t), \, t \in [0, T], \, x(\cdot) \in \Lip^0_L\big\}
    \end{equation*}
    is relatively compact. According to Proposition~\ref{prop_derivative_properties_2}, we have $\varphi(0) = 0$ for any $\varphi(\cdot) \in W.$ Hence, from Proposition~\ref{prop_p_alpha_approximation} it follows that, by the number $\varepsilon > 0,$ one can choose $h_\ast > 0$ such that, for any $h \in (0, h_\ast]$ and any $\varphi(\cdot) \in W,$ inequality (\ref{prop_p_alpha_approximation_main}) is valid. Therefore, taking Proposition~\ref{prop_p_alpha} into account, we derive (\ref{lem_RL_GL_derivatives_main}) for any $h \in (0, h_\ast]$ and any $x(\cdot) \in \Lip^0_L.$
\end{proof}

\section{\; Differential Equation of Fractional Order}\label{sec_DE}

Let $R_0 > 0$ be fixed. Let us consider the following Cauchy problem for the ordinary fractional differential equation with the Caputo derivative
\begin{equation} \label{system}
    (^C D^\alpha x) (t) = f(t, x(t)), \quad t \in [0, T], \quad x(t) \in \mathbb{R}^n,
\end{equation}
and the initial condition
\begin{equation} \label{initial_condition}
    x(0) = x_0, \quad x_0 \in B(R_0).
\end{equation}
Here the function $f: [0, T] \times \mathbb{R}^n \rightarrow \mathbb{R}^n$ satisfies the following conditions:
\begin{itemize}
  \item[($f.1$)] \ For any $x \in \mathbb{R}^n,$ the function $f(\cdot, x)$ is measurable on $[0, T].$

  \item[($f.2$)] \ For any $r > 0,$ there exists $\lambda_f > 0$ such that
        \begin{equation*}
            \|f(t, x) - f(t, y)\| \leq \lambda_f \|x - y\|,
            \quad t \in [0, T], \quad x, y \in B(r).
        \end{equation*}

  \item[($f.3$)] \ There exists $c_f > 0$ such that
        \begin{equation*}
            \|f(t, x)\| \leq (1 + \|x\|) c_f,
            \quad t \in [0, T], \quad x \in \mathbb{R}^n.
        \end{equation*}
\end{itemize}

\begin{definition}[see \cite{MIGomoyunov_2017}] \
    A function $x: [0, T] \rightarrow \mathbb{R}^n$ is called a solution of Cauchy problem {\rm(\ref{system}), (\ref{initial_condition})} if $x(\cdot) \in \{x_0\} + I^\alpha (\Linf)$ and equality {\rm(\ref{system})} holds for almost every $t \in [0, T].$
\end{definition}

Here the inclusion $x(\cdot) \in \{x_0\} + I^\alpha (\Linf)$ means that there is a function $\overline{x}(\cdot) \in I^\alpha (\Linf)$ such that $x(t) = x_0 + \overline{x}(t),$ $t \in [0, T].$ Note that $\overline{x}(0) = 0$ due to Proposition~\ref{prop_integral_properties}, and consequently, $x(0) = x_0.$ Therefore, for a function $x(\cdot) \in \{x_0\} + I^\alpha (\Linf),$ initial condition (\ref{initial_condition}) is automatically satisfied.

\begin{theorem}[see \cite{MIGomoyunov_2017}] \label{thm_existence} \
    For any initial value $x_0 \in B(R_0),$ there exists the unique solution $x(\cdot) = x(\cdot; x_0)$ of Cauchy problem {\rm(\ref{system}), (\ref{initial_condition})}. Moreover, there exist $R > 0$ and $H > 0$ such that, for any $x_0 \in B(R_0),$ the solution $x(\cdot) = x(\cdot; x_0)$ satisfies the inequalities below:
    \begin{equation*}
        \|x(t)\| \leq R, \quad
        \|x(t) - x(\tau)\| \leq H |t - \tau|^\alpha, \quad t, \tau \in [0, T].
    \end{equation*}
\end{theorem}

\section{\; Approximating Differential Equation} \label{sec_approximation_DE}

Let $x_0 \in B(R_0),$ and $x(\cdot) = x(\cdot; x_0)$ be the solution of Cauchy problem (\ref{system}), (\ref{initial_condition}). The idea of approximation of $x(\cdot)$ is the following. Let us consider the function $y(t) = (I^{1 - \alpha} (x(\cdot) - x_0))(t),$ $t \in [0, T].$ Then, due to Proposition~\ref{prop_derivative_properties}, the equality
\begin{equation}\label{idea_2}
    x(t) = x_0 + (D^{1 - \alpha} y)(t), \quad t \in [0, T],
\end{equation}
is valid. Therefore, according to (\ref{def_Caputo}) and (\ref{system}), the function $y(\cdot)$ satisfies the differential equation
\begin{equation}\label{idea_1}
    \dot{y}(t) = f\big(t, x_0 + (D^{1 - \alpha} y)(t)\big) \text{ for a.e. } t \in [0, T].
\end{equation}
Moreover, by Proposition~\ref{prop_integral_properties}, the initial condition $y(0) = 0$ holds. Now, let us fix $h > 0$ and, in accordance with Lemma~\ref{lem_RL_GL_derivatives}, approximate the fractional derivative $(D^{1 - \alpha} y)(t)$ in the right-hand sides of relations (\ref{idea_2}) and (\ref{idea_1}) by the divided fractional difference $h^{\alpha - 1} (\Delta_h^{1 - \alpha} y)(t).$ Thus, we derive the following approximating Cauchy problem for the differential equation
\begin{equation} \label{system_y_h}
    \dot{y}(t) = f\big(t, x_0 + h^{\alpha - 1} (\Delta_h^{1 - \alpha} y)(t)\big), \quad t \in [0, T], \quad y(t) \in \mathbb{R}^n,
\end{equation}
under the initial condition
\begin{equation}\label{initial_condition_y_h}
    y(0) = 0,
\end{equation}
and obtain the value $x_0 + h^{\alpha - 1} (\Delta_h^{1 - \alpha} y)(t)$ as an approximation of $x(t).$

Note that equation (\ref{system_y_h}) involves only the usual first order derivative, and its right-hand side depends on the values $y(t - jh),$ $j \in \overline{0, [t/h]}.$ Hence, this equation can be considered as a functional-differential equation of a retarded type with a finite number of delays (see, e.g., \cite{JKHale_SMVLunel_1993}).

\begin{definition} \
    A function $y: [0, T] \rightarrow \mathbb{R}^n$ is called a solution of Cauchy problem {\rm(\ref{system_y_h}), (\ref{initial_condition_y_h})} if $y(\cdot) \in \Lip^0$ and equality {\rm(\ref{system_y_h})} holds for almost every $t \in [0, T].$
\end{definition}

Applying the successive integration method (step-by-step method) (see, e.g., \cite[\S1.2]{AMZverkin_GAKemenskii_SBNorkin_LEElsgolts_1962}), one can show that, due to conditions $(f.1)$--$(f.3),$ the following proposition is valid (see also \cite[Ch.~2]{JKHale_SMVLunel_1993}).
\begin{proposition} \label{prop_existence_y_h} \
    For any $x_0 \in B(R_0)$ and any $h > 0,$ there exists the unique solution $y(\cdot) = y(\cdot; x_0; h)$ of Cauchy problem {\rm(\ref{system_y_h}), (\ref{initial_condition_y_h})}.
\end{proposition}

\begin{lemma} \label{lem_properties_approximating_system} \
    There exists $L > 0$ such that, for any $x_0 \in B(R_0)$ and any $h > 0,$ the solution $y(\cdot) = y(\cdot; x_0; h)$ of Cauchy problem $(\ref{system_y_h}),$ $(\ref{initial_condition_y_h})$ satisfies the inclusion $y(\cdot) \in \Lip_L^0.$
\end{lemma}
\begin{proof} \
    Let $c_f$ be the constant from ($f.3$) and $\|p_{1-\alpha}(\cdot)\|_1$ be defined by (\ref{p_alpha_properties}). Let us denote $a = \max \{1 + R_0, \|p_{1 - \alpha}(\cdot)\|_1\} c_f$ and put $L = E_\alpha(a T^\alpha) a,$ where $E_\alpha(\cdot)$ is the Mittag-Leffler function (see, e.g., \cite[(1.90)]{SGSamko_AAKilbas_OIMarichev_1993}). Let us show that this $L$ satisfies the statement of the lemma.

    Let $x_0 \in B(R_0),$ $h > 0$ and $y(\cdot) = y(\cdot; x_0; h)$ be the solution of Cauchy problem (\ref{system_y_h}), (\ref{initial_condition_y_h}). Since $y(\cdot) \in \Lip^0,$ to prove the inclusion $y(\cdot) \in \Lip^0_L,$ it is sufficient to show that $\|\dot{y}(t)\| \leq L$ for almost every $t \in [0, T].$ Let $\varphi(t) = (D^{1 - \alpha} y)(t),$ $t \in [0, T].$ Due to Proposition~\ref{prop_derivative_properties_2} and (\ref{system_y_h}), we have
    \begin{equation*} \label{lem_pas_proof_r}
        \varphi(t) = \frac{1}{\Gamma(\alpha)} \int_{0}^{t} \frac{f \big(t, x_0 + h^{\alpha - 1} (\Delta_h^{1 - \alpha} y) (\tau)\big)}
        {(t - \tau)^{1 - \alpha}} d\tau, \quad t \in [0, T].
    \end{equation*}
    According to Proposition~\ref{prop_p_alpha}, we obtain
    \begin{equation*}
        \| h^{\alpha - 1} (\Delta_h^{1 - \alpha} y) (t) \|
        \leq \|p_{1 - \alpha}(\cdot)\|_1 \max_{\xi \in [0, t]} \|\varphi (\xi)\|, \quad t \in [0, T].
    \end{equation*}
    Hence, by $(f.3),$ we derive
    \begin{equation*} \label{lem_pas_proof_r_estimate}
        \|\varphi (t)\|
        \leq \frac{1}{\Gamma(\alpha)} \int_{0}^{t} \frac{c_f\big(1 + \|x_0\| + \|h^{\alpha - 1} (\Delta_h^{1 - \alpha} y) (\tau)\| \big)}{(t - \tau)^{1 - \alpha}} d \tau
        \leq \frac{a}{\Gamma(\alpha)} \int_{0}^{t} \frac{1 + \max_{\xi \in [0, \tau]} \|\varphi (\xi)\|}{(t - \tau)^{1 - \alpha}} d \tau,
        \quad t \in [0, T].
    \end{equation*}
    Since, due to Proposition~\ref{Prop_integral_properties_2}, the function from the right-hand side of these inequalities is non-decreasing in $t \in [0, T],$ then
    \begin{equation*}
        1 + \max_{\xi \in [0, t]} \|\varphi (\xi)\|
        \leq 1 + \frac{a}{\Gamma(\alpha)} \int_{0}^{t} \frac{1 + \max_{\xi \in [0, \tau]} \|\varphi (\xi)\|}{(t - \tau)^{1 - \alpha}} d \tau,
        \quad t \in [0, T],
    \end{equation*}
    wherefrom, applying the fractional version of Bellman-Gronwall lemma (see, e.g., \cite[Lemma~6.19]{KDiethelm_2010} and also \cite[Lemma~1.1]{MIGomoyunov_2017}), we conclude
    \begin{equation*}
        1 + \max_{\xi \in [0, t]} \|\varphi(\xi)\| \leq E_\alpha (a T^\alpha), \quad t \in [0, T].
    \end{equation*}
    Thus, according to (\ref{system_y_h}) and $(f.3),$ we have
    \begin{equation*}
        \|\dot{y}(t)\| \leq c_f\big(1 + \|x_0\| + \|h^{\alpha - 1} (\Delta_h^{1 - \alpha} y) (t)\| \big)
        \leq a \big(1 + \max_{\xi \in [0, t]} \|\varphi (\xi)\|\big)
        \leq a E_\alpha (a T^\alpha) = L \text{ for a.e. } t \in [0, T].
    \end{equation*}
    The lemma is proved.
\end{proof}

\begin{theorem} \label{thm_approximation} \
    For any $\varepsilon > 0,$ there exists $h_\ast > 0$ such that, for any initial value $x_0 \in B(R_0)$ and any $h \in (0, h_\ast],$ the solutions $x(\cdot) = x(\cdot; x_0)$ of Cauchy problem $(\ref{system}),$ $(\ref{initial_condition})$ and $y(\cdot) = y(\cdot; x_0; h)$ of approximating Cauchy problem $(\ref{system_y_h}),$ $(\ref{initial_condition_y_h})$ satisfy the inequality below:
    \begin{equation} \label{thm_approximation_main}
        \|x(t) - x_0 - h^{\alpha - 1} (\Delta_h^{1 - \alpha} y)(t)\| \leq \varepsilon, \quad t \in [0, T].
    \end{equation}
\end{theorem}
\begin{proof} \
    According to Propositions~\ref{prop_RL_derivative_Lipshitz} and~\ref{prop_GL_derivative_Lipshitz}, by the number $L > 0$ from Lemma~\ref{lem_properties_approximating_system}, let us choose $M > 0$ and $\overline{M} > 0$ such that, for any $h > 0$ and any $y(\cdot) \in \Lip_L^0,$ the following inequalities are valid:
    \begin{equation} \label{thm_a_p_M}
        \|(D^{1 - \alpha} y) (t)\| \leq M, \quad
        \|h^{\alpha - 1} (\Delta_h^{1 - \alpha} y)(t)\| \leq \overline{M}, \quad t \in [0, T].
    \end{equation}
    Let $R > 0$ be taken from Theorem~\ref{thm_existence} and $R_1 = R + R_0 + M + \overline{M}.$ By the number $R_1,$ let us choose $\lambda_f > 0$ according to ($f.2$). Let $\varepsilon > 0$ be fixed. Let $\eta > 0$ and $\zeta > 0$ satisfy the inequalities
    \begin{equation} \label{thm_a_p_eta_zeta}
        \eta \leq \Gamma(\alpha + 1) \varepsilon^2 / (4 T^\alpha E_\alpha (2 \lambda_f T^\alpha) ), \quad
        \zeta \leq \eta/(2 \lambda_f R_1).
    \end{equation}
    Let $h_\ast > 0$ be chosen by Lemma~\ref{lem_RL_GL_derivatives} such that, for any $h \in (0, h_\ast]$ and any $y(\cdot) \in \Lip^0_L,$ the following inequality holds:
    \begin{equation} \label{thm_a_p_h_ast}
        \|(D^{1 - \alpha} y)(t) - h^{\alpha - 1} (\Delta_h^{1 - \alpha} y) (t)\| \leq \min\big\{ \zeta,  \varepsilon/2 \big\},
        \quad t \in [0, T].
    \end{equation}
    Let us show that this $h_\ast$ satisfies the statement of the theorem.

    Let $x_0 \in B(R_0)$ and $h \in (0, h_\ast].$ Let $x(\cdot) = x(\cdot; x_0)$ and $y(\cdot) = y(\cdot; x_0; h)$ be, respectively, the solutions of Cauchy problems (\ref{system}), (\ref{initial_condition}) and (\ref{system_y_h}), (\ref{initial_condition_y_h}).
    Note that, by the choice of $L,$ the inclusion $y(\cdot) \in \Lip_L^0$ holds, and therefore, relations (\ref{thm_a_p_M}) and (\ref{thm_a_p_h_ast}) are valid. Let us consider the function
    \begin{equation} \label{thm_a_p_s}
        s(t) = x(t) - x_0 - (D^{1 - \alpha} y)(t), \quad t \in [0, T].
    \end{equation}
    According to Proposition~\ref{prop_derivative_properties_2}, we have $(D^{1 - \alpha} y)(\cdot) \in I^\alpha(\Linf)$ and $\dot{y}(t) = (D^\alpha (D^{1 - \alpha} y))(t)$ for almost every $t \in [0, T].$ Hence, taking into account that $x(\cdot) \in \{x_0\} + I^\alpha(\Linf),$ we obtain $s(\cdot) \in I^\alpha(\Linf).$ Furthermore, according to (\ref{system}), (\ref{system_y_h}), we have
    \begin{equation*}
        (D^\alpha s) (t) = f(t, x(t)) - f(t, \widetilde{x}(t)) \text{ for a.e. } t \in [0, T],
    \end{equation*}
    where, for brevity, we denote
    \begin{equation} \label{thm_a_p_widetilde_x}
        \widetilde{x}(t) = x_0 + h^{\alpha - 1} (\Delta_h^{1 - \alpha} y)(t), \quad t \in [0, T].
    \end{equation}
    Let us consider the Lyapunov function
    \begin{equation} \label{thm_a_p_V}
        V(t) = \|s(t)\|^2, \quad t \in [0, T].
    \end{equation}
    Applying \cite[Corollary~3.2]{MIGomoyunov_2017}, we derive $V(\cdot) \in I^\alpha(\Linf([0, T], \mathbb{R}))$ and
    \begin{equation} \label{thm_a_p_DV_estimate}
        (D^\alpha V)(t) \leq 2 \langle s(t), (D^\alpha s)(t) \rangle
        = 2 \langle s(t), f(t, x(t)) - f(t, \widetilde{x}(t) \rangle \text{ for a.e. } t \in [0, T].
    \end{equation}
    Let us estimate the right-hand side of these relations for $t \in [0, T].$ Note that, by the choice of $R_1,$ we have
    \begin{equation*}
        \max \big\{ \|x(t)\|, \|\widetilde{x}(t)\|, \|s(t)\| \big\} \leq R_1.
    \end{equation*}
    Therefore, by the choice of $\lambda_f,$ $h_\ast$ and $\zeta,$ we obtain
    \begin{equation} \label{thm_a_p_product_estimate}
        \begin{array}{c}
            \langle s(t), f(t, x(t)) - f(t, \widetilde{x}(t) \rangle
            \leq \lambda_f \|s(t)\| \|x(t) - \widetilde{x}(t)\| 
            \leq \lambda_f \|s(t)\| \big( \|s(t)\| + \|(D^{1 - \alpha}y)(t) - (\Delta_h^{1 - \alpha}y)(t)\| \big) \\[0.5em]
            \leq \lambda_f V(t) + \lambda_f R_1 \zeta
            \leq \lambda_f V(t) + \eta/2.
        \end{array}
    \end{equation}
    From (\ref{thm_a_p_DV_estimate}) and (\ref{thm_a_p_product_estimate}) it follows that
    \begin{equation} \label{thm_a_p_DV_estimate_final}
        (D^\alpha V)(t) \leq 2 \lambda_f V(t) + \eta \text{ for a.e. } t \in [0, T].
    \end{equation}
    Consequently, basing on Proposition~\ref{prop_integral_properties}, for $t \in [0, T],$ we obtain
    \begin{equation*}
        V(t) \leq \frac{1}{\Gamma(\alpha)} \int_{0}^{t} \frac{2 \lambda_f V(\tau) + \eta}{(t - \tau)^{1 - \alpha}} d \tau \\[0.5em]
        \leq \frac{\eta T^\alpha}{\Gamma(\alpha + 1)}
        + \frac{2 \lambda_f}{\Gamma(\alpha)} \int_{0}^{t} \frac{V(\tau)}{(t - \tau)^{1 - \alpha}} d \tau,
    \end{equation*}
    wherefrom, applying the fractional version of Bellman-Gronwall lemma, due to the choice of $\eta,$ we deduce
    \begin{equation*}
        V(t) \leq \eta T^\alpha E_\alpha(2 \lambda_f T^\alpha) / \Gamma(\alpha + 1) \leq \varepsilon^2/4, \quad t \in [0, T].
    \end{equation*}
    Thus, for $t \in [0, T],$ we have
    \begin{equation*}
        \|x(t) - x_0 - (D^{1 - \alpha} y) (t)\| \leq \varepsilon/2,
    \end{equation*}
    and therefore, by the choice of $h_\ast,$ we derive
    \begin{equation*}
        \|x(t) - x_0 - h^{\alpha - 1}(\Delta_h^{1 - \alpha} y)(t)\|
        \leq \varepsilon/2 + \|(D^{1 - \alpha} y)(t) - h^{\alpha - 1} (\Delta_h^{1 - \alpha} y)(t)\|
        \leq \varepsilon.
    \end{equation*}
    The theorem is proved.
\end{proof}

\section{\; Example~1}\label{sec_Example_1}

Let us illustrate Theorem~\ref{thm_approximation} by an example. Let us consider the following Cauchy problem for the system of fractional order differential equations
\begin{equation} \label{system_ex}
    \begin{array}{c}
        \begin{cases}
            (^C D^{0.3} x_1) (t) = x_1(t) - x_2(t) + \cos(2 \, t), \\[0.5em]
            (^C D^{0.3} x_2) (t) = t \, x_1(t) + e^{\cos(x_2(t))} + \sin(2 \, t),
        \end{cases} \\[1.7em]
        t \in [0, 5], \quad x(t) = (x_1(t), x_2(t)) \in \mathbb{R}^2,
    \end{array}
\end{equation}
with the initial condition
\begin{equation}\label{initial_condition_ex}
    x(0) = x_0 = (0.5, -1),
\end{equation}
and the corresponding approximating Cauchy problem (\ref{system_y_h}), (\ref{initial_condition_y_h}). For the numerical solution of these problems, we use the forward Euler methods (see, e.g., \cite[p.~101]{CLi_FZeng_2015} and \cite[p.~115]{AVKim_2015}) with the constant step $0.001.$

The following three cases were considered. In the first one, we choose $h = 0.1.$ The obtained difference between the solution $x(\cdot)$ of (\ref{system_ex}), (\ref{initial_condition_ex}) and its approximation $\widetilde{x}(t) = x_0 + h^{-0.7} (\Delta_h^{0.7} y)(t),$ $t \in [0, 5],$ is
\begin{equation*}
    \max_{t \in [0, 5]} \|x(t) - \widetilde{x}(t)\| \approx 0.9585.
\end{equation*}
In the second case, we choose $h = 0.01$ and obtain
\begin{equation*}
    \max_{t \in [0, 5]} \|x(t) - \widetilde{x}(t)\| \approx 0.4232.
\end{equation*}
In the third case, for $h = 0.001,$ we have
\begin{equation*}
    \max_{t \in [0, 5]} \|x(t) - \widetilde{x}(t)\| \approx 0.0436.
\end{equation*}
The simulations results are shown in Fig.~4.1.

\vspace*{-2.5em}

\begin{center}
    \hspace*{-2em}
\begingroup
  \makeatletter
  \providecommand\color[2][]{%
    \GenericError{(gnuplot) \space\space\space\@spaces}{%
      Package color not loaded in conjunction with
      terminal option `colourtext'%
    }{See the gnuplot documentation for explanation.%
    }{Either use 'blacktext' in gnuplot or load the package
      color.sty in LaTeX.}%
    \renewcommand\color[2][]{}%
  }%
  \providecommand\includegraphics[2][]{%
    \GenericError{(gnuplot) \space\space\space\@spaces}{%
      Package graphicx or graphics not loaded%
    }{See the gnuplot documentation for explanation.%
    }{The gnuplot epslatex terminal needs graphicx.sty or graphics.sty.}%
    \renewcommand\includegraphics[2][]{}%
  }%
  \providecommand\rotatebox[2]{#2}%
  \@ifundefined{ifGPcolor}{%
    \newif\ifGPcolor
    \GPcolorfalse
  }{}%
  \@ifundefined{ifGPblacktext}{%
    \newif\ifGPblacktext
    \GPblacktexttrue
  }{}%
  \let\gplgaddtomacro\g@addto@macro
  \gdef\gplbacktext{}%
  \gdef\gplfronttext{}%
  \makeatother
  \ifGPblacktext
    \def\colorrgb#1{}%
    \def\colorgray#1{}%
  \else
    \ifGPcolor
      \def\colorrgb#1{\color[rgb]{#1}}%
      \def\colorgray#1{\color[gray]{#1}}%
      \expandafter\def\csname LTw\endcsname{\color{white}}%
      \expandafter\def\csname LTb\endcsname{\color{black}}%
      \expandafter\def\csname LTa\endcsname{\color{black}}%
      \expandafter\def\csname LT0\endcsname{\color[rgb]{1,0,0}}%
      \expandafter\def\csname LT1\endcsname{\color[rgb]{0,1,0}}%
      \expandafter\def\csname LT2\endcsname{\color[rgb]{0,0,1}}%
      \expandafter\def\csname LT3\endcsname{\color[rgb]{1,0,1}}%
      \expandafter\def\csname LT4\endcsname{\color[rgb]{0,1,1}}%
      \expandafter\def\csname LT5\endcsname{\color[rgb]{1,1,0}}%
      \expandafter\def\csname LT6\endcsname{\color[rgb]{0,0,0}}%
      \expandafter\def\csname LT7\endcsname{\color[rgb]{1,0.3,0}}%
      \expandafter\def\csname LT8\endcsname{\color[rgb]{0.5,0.5,0.5}}%
    \else
      \def\colorrgb#1{\color{black}}%
      \def\colorgray#1{\color[gray]{#1}}%
      \expandafter\def\csname LTw\endcsname{\color{white}}%
      \expandafter\def\csname LTb\endcsname{\color{black}}%
      \expandafter\def\csname LTa\endcsname{\color{black}}%
      \expandafter\def\csname LT0\endcsname{\color{black}}%
      \expandafter\def\csname LT1\endcsname{\color{black}}%
      \expandafter\def\csname LT2\endcsname{\color{black}}%
      \expandafter\def\csname LT3\endcsname{\color{black}}%
      \expandafter\def\csname LT4\endcsname{\color{black}}%
      \expandafter\def\csname LT5\endcsname{\color{black}}%
      \expandafter\def\csname LT6\endcsname{\color{black}}%
      \expandafter\def\csname LT7\endcsname{\color{black}}%
      \expandafter\def\csname LT8\endcsname{\color{black}}%
    \fi
  \fi
    \setlength{\unitlength}{0.0500bp}%
    \ifx\gptboxheight\undefined%
      \newlength{\gptboxheight}%
      \newlength{\gptboxwidth}%
      \newsavebox{\gptboxtext}%
    \fi%
    \setlength{\fboxrule}{0.5pt}%
    \setlength{\fboxsep}{1pt}%
\begin{picture}(4752.00,5760.00)%
    \gplgaddtomacro\gplbacktext{%
      \csname LTb\endcsname%
      \put(462,1087){\makebox(0,0)[r]{\strut{}$-3$}}%
      \csname LTb\endcsname%
      \put(462,1714){\makebox(0,0)[r]{\strut{}$-2$}}%
      \csname LTb\endcsname%
      \put(462,2341){\makebox(0,0)[r]{\strut{}$-1$}}%
      \csname LTb\endcsname%
      \put(462,2968){\makebox(0,0)[r]{\strut{}$0$}}%
      \csname LTb\endcsname%
      \put(462,3594){\makebox(0,0)[r]{\strut{}$1$}}%
      \csname LTb\endcsname%
      \put(462,4221){\makebox(0,0)[r]{\strut{}$2$}}%
      \csname LTb\endcsname%
      \put(462,4848){\makebox(0,0)[r]{\strut{}$x$}}%
      \csname LTb\endcsname%
      \put(594,867){\makebox(0,0){\strut{}$0$}}%
      \csname LTb\endcsname%
      \put(1346,867){\makebox(0,0){\strut{}$1$}}%
      \csname LTb\endcsname%
      \put(2098,867){\makebox(0,0){\strut{}$2$}}%
      \csname LTb\endcsname%
      \put(2851,867){\makebox(0,0){\strut{}$3$}}%
      \csname LTb\endcsname%
      \put(3603,867){\makebox(0,0){\strut{}$4$}}%
      \csname LTb\endcsname%
      \put(4355,867){\makebox(0,0){\strut{}$t$}}%
    }%
    \gplgaddtomacro\gplfronttext{%
      \csname LTb\endcsname%
      \put(3032,4403){\makebox(0,0)[r]{\strut{}$x_1$}}%
      \csname LTb\endcsname%
      \put(3032,4139){\makebox(0,0)[r]{\strut{}$x_2$}}%
      \csname LTb\endcsname%
      \put(3979,4403){\makebox(0,0)[r]{\strut{}$\widetilde{x}_1$}}%
      \csname LTb\endcsname%
      \put(3979,4139){\makebox(0,0)[r]{\strut{}$\widetilde{x}_2$}}%
    }%
    \gplbacktext
    \put(0,0){\includegraphics{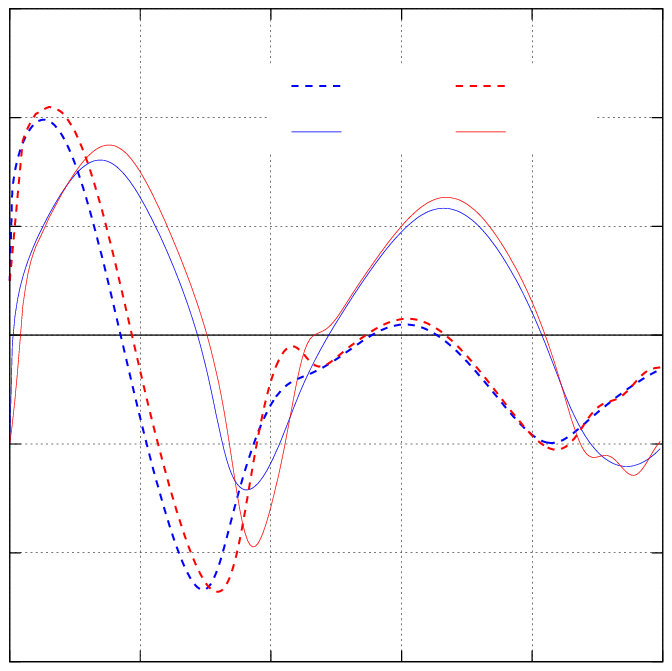}}%
    \gplfronttext
  \end{picture}%
\endgroup

    \hspace*{-1em}
\begingroup
  \makeatletter
  \providecommand\color[2][]{%
    \GenericError{(gnuplot) \space\space\space\@spaces}{%
      Package color not loaded in conjunction with
      terminal option `colourtext'%
    }{See the gnuplot documentation for explanation.%
    }{Either use 'blacktext' in gnuplot or load the package
      color.sty in LaTeX.}%
    \renewcommand\color[2][]{}%
  }%
  \providecommand\includegraphics[2][]{%
    \GenericError{(gnuplot) \space\space\space\@spaces}{%
      Package graphicx or graphics not loaded%
    }{See the gnuplot documentation for explanation.%
    }{The gnuplot epslatex terminal needs graphicx.sty or graphics.sty.}%
    \renewcommand\includegraphics[2][]{}%
  }%
  \providecommand\rotatebox[2]{#2}%
  \@ifundefined{ifGPcolor}{%
    \newif\ifGPcolor
    \GPcolorfalse
  }{}%
  \@ifundefined{ifGPblacktext}{%
    \newif\ifGPblacktext
    \GPblacktexttrue
  }{}%
  \let\gplgaddtomacro\g@addto@macro
  \gdef\gplbacktext{}%
  \gdef\gplfronttext{}%
  \makeatother
  \ifGPblacktext
    \def\colorrgb#1{}%
    \def\colorgray#1{}%
  \else
    \ifGPcolor
      \def\colorrgb#1{\color[rgb]{#1}}%
      \def\colorgray#1{\color[gray]{#1}}%
      \expandafter\def\csname LTw\endcsname{\color{white}}%
      \expandafter\def\csname LTb\endcsname{\color{black}}%
      \expandafter\def\csname LTa\endcsname{\color{black}}%
      \expandafter\def\csname LT0\endcsname{\color[rgb]{1,0,0}}%
      \expandafter\def\csname LT1\endcsname{\color[rgb]{0,1,0}}%
      \expandafter\def\csname LT2\endcsname{\color[rgb]{0,0,1}}%
      \expandafter\def\csname LT3\endcsname{\color[rgb]{1,0,1}}%
      \expandafter\def\csname LT4\endcsname{\color[rgb]{0,1,1}}%
      \expandafter\def\csname LT5\endcsname{\color[rgb]{1,1,0}}%
      \expandafter\def\csname LT6\endcsname{\color[rgb]{0,0,0}}%
      \expandafter\def\csname LT7\endcsname{\color[rgb]{1,0.3,0}}%
      \expandafter\def\csname LT8\endcsname{\color[rgb]{0.5,0.5,0.5}}%
    \else
      \def\colorrgb#1{\color{black}}%
      \def\colorgray#1{\color[gray]{#1}}%
      \expandafter\def\csname LTw\endcsname{\color{white}}%
      \expandafter\def\csname LTb\endcsname{\color{black}}%
      \expandafter\def\csname LTa\endcsname{\color{black}}%
      \expandafter\def\csname LT0\endcsname{\color{black}}%
      \expandafter\def\csname LT1\endcsname{\color{black}}%
      \expandafter\def\csname LT2\endcsname{\color{black}}%
      \expandafter\def\csname LT3\endcsname{\color{black}}%
      \expandafter\def\csname LT4\endcsname{\color{black}}%
      \expandafter\def\csname LT5\endcsname{\color{black}}%
      \expandafter\def\csname LT6\endcsname{\color{black}}%
      \expandafter\def\csname LT7\endcsname{\color{black}}%
      \expandafter\def\csname LT8\endcsname{\color{black}}%
    \fi
  \fi
    \setlength{\unitlength}{0.0500bp}%
    \ifx\gptboxheight\undefined%
      \newlength{\gptboxheight}%
      \newlength{\gptboxwidth}%
      \newsavebox{\gptboxtext}%
    \fi%
    \setlength{\fboxrule}{0.5pt}%
    \setlength{\fboxsep}{1pt}%
\begin{picture}(4752.00,5760.00)%
    \gplgaddtomacro\gplbacktext{%
      \csname LTb\endcsname%
      \put(462,1087){\makebox(0,0)[r]{\strut{}$-3$}}%
      \csname LTb\endcsname%
      \put(462,1714){\makebox(0,0)[r]{\strut{}$-2$}}%
      \csname LTb\endcsname%
      \put(462,2341){\makebox(0,0)[r]{\strut{}$-1$}}%
      \csname LTb\endcsname%
      \put(462,2968){\makebox(0,0)[r]{\strut{}$0$}}%
      \csname LTb\endcsname%
      \put(462,3594){\makebox(0,0)[r]{\strut{}$1$}}%
      \csname LTb\endcsname%
      \put(462,4221){\makebox(0,0)[r]{\strut{}$2$}}%
      \csname LTb\endcsname%
      \put(462,4848){\makebox(0,0)[r]{\strut{}$x$}}%
      \csname LTb\endcsname%
      \put(594,867){\makebox(0,0){\strut{}$0$}}%
      \csname LTb\endcsname%
      \put(1346,867){\makebox(0,0){\strut{}$1$}}%
      \csname LTb\endcsname%
      \put(2098,867){\makebox(0,0){\strut{}$2$}}%
      \csname LTb\endcsname%
      \put(2851,867){\makebox(0,0){\strut{}$3$}}%
      \csname LTb\endcsname%
      \put(3603,867){\makebox(0,0){\strut{}$4$}}%
      \csname LTb\endcsname%
      \put(4355,867){\makebox(0,0){\strut{}$t$}}%
    }%
    \gplgaddtomacro\gplfronttext{%
      \csname LTb\endcsname%
      \put(3032,4403){\makebox(0,0)[r]{\strut{}$x_1$}}%
      \csname LTb\endcsname%
      \put(3032,4139){\makebox(0,0)[r]{\strut{}$x_2$}}%
      \csname LTb\endcsname%
      \put(3979,4403){\makebox(0,0)[r]{\strut{}$\widetilde{x}_1$}}%
      \csname LTb\endcsname%
      \put(3979,4139){\makebox(0,0)[r]{\strut{}$\widetilde{x}_2$}}%
    }%
    \gplbacktext
    \put(0,0){\includegraphics{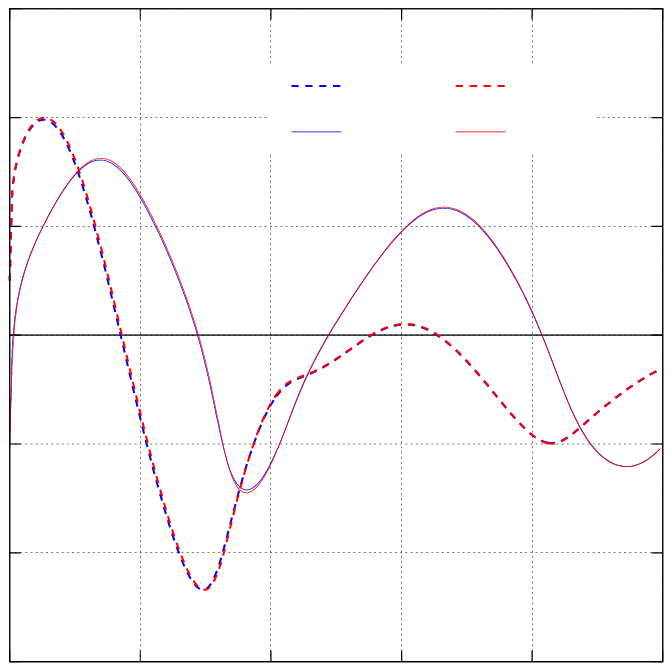}}%
    \gplfronttext
  \end{picture}%
\endgroup

    \hspace*{-2em}

    \vspace*{-1.8em}

    Fig. 4.1: The obtained solution $x(\cdot)$ of Cauchy problem (\ref{system_ex}), (\ref{initial_condition_ex}) and its approximations $\widetilde{x}(\cdot)$ for $h = 0.1$ and $h = 0.01.$
\end{center}

\section{\; Conflict-Controlled Dynamical System of Fractional Order and its Approximation}\label{sec_DS}

Let us consider a conflict-controlled dynamical system which motion is described by the fractional differential equation
\begin{equation} \label{system_u_v}
    (^C D^\alpha x) (t) = g(t, x(t), u(t), v(t)), \quad t \in [0, T], \quad
    x(t) \in \mathbb{R}^n, \quad u(t) \in P \subset \mathbb{R}^{n_u}, \quad v(t) \in Q \subset \mathbb{R}^{n_v},
\end{equation}
with the initial condition
\begin{equation}\label{initial_condition_u_v}
    x(0) = x_0, \quad x_0 \in B(R_0).
\end{equation}
Here $t$ is the time variable, $x$ is the state vector, $u$ is the control vector and $v$ is the vector of unknown disturbances; $n_u, n_v \in \mathbb{N};$ $P$ and $Q$ are compact sets; $x_0$ is the initial value of the state vector. The function $g: [0, T] \times \mathbb{R}^n \times P \times Q \rightarrow \mathbb{R}^n$ satisfies the following conditions:
\begin{itemize}
  \item[($g.1$)] \ The function $g(\cdot)$ is continuous.
  \item[($g.2$)] \ For any $r \geq 0,$ there exists $\lambda_g > 0$ such that
      \begin{equation*}
        \|g(t, x, u, v) - g(t, y, u, v)\| \leq \lambda_g \|x - y\|, \quad
        t \in [0, T], \quad x, y \in B(r), \quad u \in P, \quad v \in Q.
      \end{equation*}
  \item[($g.3$)] \ There exits $c_g > 0$ such that
    \begin{equation*}
        \|g(t, x, u, v)\| \leq (1 + \|x\|) c_g,
        \ \ t \in [0, T], \ \ x \in \mathbb{R}^n, \ \ u \in P, \ \ v \in Q.
    \end{equation*}
  \item[($g.4$)] \ For any $t \in [0, T]$ and $x, s \in \mathbb{R}^n,$ the following equality holds:
    \begin{equation*}
            \displaystyle
            \min_{u \in P} \max_{v \in Q} \langle s, g(t, x, u, v) \rangle
            = \max_{v \in Q} \min_{u \in P} \langle s, g(t, x, u, v) \rangle.
    \end{equation*}
\end{itemize}
It should be noted here that these conditions are quite standard for control problems under disturbances and differential games (see, e.g., \cite[pp.~7, 8]{NNKrasovskii_AISubbotin_1988}).

By admissible control and disturbance realizations, we mean measurable functions $u: [0, T) \rightarrow P$ and $v: [0, T) \rightarrow Q,$ respectively. The sets of all admissible control $u(\cdot)$ and disturbance $v(\cdot)$ realizations are denoted by $\mathcal{U}$ and $\mathcal{V}.$ A motion of system (\ref{system_u_v}), (\ref{initial_condition_u_v}) that corresponds to an initial value $x_0 \in B(R_0)$ and realizations $u(\cdot) \in \mathcal{U},$ $v(\cdot) \in \mathcal{V}$ is a function $x(\cdot) \in \{x_0\} + I^\alpha (\Linf)$ which, together with $u(\cdot)$ and $v(\cdot),$ satisfies (\ref{system_u_v}) for almost every $t \in [0, T].$

\begin{proposition}[see \cite{MIGomoyunov_2017}] \label{Prop_motion_u_v_properties} \
    For any initial value $x_0 \in B(R_0)$ and any realizations $u(\cdot) \in \mathcal{U},$ $v(\cdot) \in \mathcal{V},$ there exists the unique motion $x(\cdot) = x(\cdot; x_0; u(\cdot), v(\cdot))$ of system $(\ref{system_u_v}),$ $(\ref{initial_condition_u_v}).$ Moreover, there exist $\overline{R} > 0$ and $\overline{H} > 0$ such that, for any $x_0 \in B(R_0),$ $u(\cdot) \in \mathcal{U},$ $v(\cdot) \in \mathcal{V},$ the motion $x(\cdot) = x(\cdot; x_0; u(\cdot), v(\cdot))$ satisfies the inequalities below:
    \begin{equation*}
        \|x(t)\| \leq \overline{R}, \quad \|x(t) - x(\tau)\| \leq \overline{H} |t - \tau|^\alpha, \quad t, \tau \in [0, T].
    \end{equation*}
\end{proposition}

Let $x_0 \in B(R_0)$ and $h > 0$ be fixed. In accordance with Sect.~\ref{sec_approximation_DE}, let us consider the following approximating system
\begin{equation} \label{system_y_u_v}
    \dot{y}(t) = g\big(t, x_0 + h^{\alpha - 1} (\Delta_h^{1 - \alpha} y) (t), \widetilde{u}(t), \widetilde{v}(t)\big),
    \quad t \in [0, T], \quad
    y(t) \in \mathbb{R}^n, \quad \widetilde{u}(t) \in P, \quad \widetilde{v}(t) \in Q,
\end{equation}
with the initial condition
\begin{equation}\label{initial_condition_y_u_v}
    y(0) = 0.
\end{equation}
Here $y$ is the state vector of the approximating system, $\widetilde{u}$ and $\widetilde{v}$ are control vectors. By a motion of system (\ref{system_y_u_v}), (\ref{initial_condition_y_u_v}) that corresponds to admissible realizations $\widetilde{u}(\cdot) \in \mathcal{U},$ $\widetilde{v}(\cdot) \in \mathcal{V},$ we mean a function $y(\cdot) \in \Lip^0$ which, together with $\widetilde{u}(\cdot)$ and $\widetilde{v}(\cdot),$ satisfies (\ref{system_y_u_v}) for almost every $t \in [0, T].$

By analogy with Proposition~\ref{prop_existence_y_h} and Lemma~\ref{lem_properties_approximating_system}, one can prove the following proposition.
\begin{proposition} \label{Prop_motion_y_u_v_properties} \
    For any $x_0 \in B(R_0),$ $h > 0,$ $\widetilde{u}(\cdot) \in \mathcal{U},$ $\widetilde{v}(\cdot) \in \mathcal{V},$ there exists the unique motion $y(\cdot) = y(\cdot; x_0; h; \widetilde{u}(\cdot), \widetilde{v}(\cdot))$ of approximating system $(\ref{system_y_u_v}),$ $(\ref{initial_condition_y_u_v}).$ Moreover, there exists $\overline{L} > 0$ such that, for any $x_0 \in B(R_0),$ $h > 0,$ $\widetilde{u}(\cdot) \in \mathcal{U},$ $\widetilde{v}(\cdot) \in \mathcal{V},$ the motion $y(\cdot) = y(\cdot; x_0; h; \widetilde{u}(\cdot), \widetilde{v}(\cdot))$ satisfies the inclusion $y(\cdot) \in \Lip^0_{\overline{L}}.$
\end{proposition}

Note that, since the right-hand side of system (\ref{system_u_v}) contains unknown disturbances $v(t),$ then the results from Sect.~\ref{sec_approximation_DE} can not be directly applied here. In order to overcome this difficulty, in the next section, we propose a mutual aiming procedure between initial (\ref{system_u_v}), (\ref{initial_condition_u_v}) and approximating (\ref{system_y_u_v}), (\ref{initial_condition_y_u_v}) systems. This procedure is based on the extremal shift rule (see, e.g., \cite[\S\S~2.4, 8.2]{NNKrasovskii_AISubbotin_1988} and also \cite{MIGomoyunov_2017}) and specifies the way of forming control realizations $u(\cdot) \in \mathcal{U}$ and $\widetilde{v}(\cdot) \in \mathcal{V}$ that guarantees estimate (\ref{thm_approximation_main}) for any disturbance realization $v(\cdot) \in \mathcal{V}$ and any control realization $\widetilde{u}(\cdot) \in \mathcal{U}.$

\section{\; Mutual Aiming Procedure}\label{sec_Procedure}

Let $x_0 \in B(R_0),$ $h > 0$ and
\begin{equation*}
    \Delta = \{\tau_j\}_{j = 1}^{k + 1} \subset [0, T], \quad
    \tau_1 = 0, \quad  \tau_{j+1} > \tau_j, \quad j \in \overline{1, k}, \quad \tau_{k+1} = T, \quad k \in \mathbb{N},
\end{equation*}
be a partition of the segment $[0, T].$ Let us consider the following procedure of forming realizations $u(\cdot) \in \mathcal{U}$ in system (\ref{system_u_v}), (\ref{initial_condition_u_v}) and $\widetilde{v}(\cdot) \in \mathcal{V}$ in approximating system (\ref{system_y_u_v}), (\ref{initial_condition_y_u_v}). Let $j \in \overline{1, k}$ and the values $x(\tau_j)$ in system (\ref{system_u_v}) and $y(t),$ $t \in [0, \tau_j],$ in system (\ref{system_y_u_v}) have already been realized. Then, for $t \in [\tau_j, \tau_{j+1}),$ we define
\begin{equation}\label{procedure}
    \begin{array}{c}
        \begin{array}{rcl}
            u(t) = u_j & \in &
            \displaystyle \argmin{u \in P} \max_{v \in Q}
            \langle \widetilde{s}_j, g(\tau_j, x(\tau_j), u, v) \rangle, \\[0.5em]
            \widetilde{v}(t) = \widetilde{v}_j & \in &
            \displaystyle \argmax{\widetilde{v} \in Q} \min_{\widetilde{u} \in P}
            \langle \widetilde{s}_j, g\big(\tau_j, x_0 + h^{\alpha - 1} (\Delta_h^{1 - \alpha} y) (\tau_j), \widetilde{u}, \widetilde{v}\big) \rangle,
        \end{array} \\
        \widetilde{s}_j = x(\tau_j) - x_0 - h^{\alpha - 1} (\Delta_h^{1 - \alpha} y) (\tau_j).
    \end{array}
\end{equation}

\begin{theorem} \label{thm_closeness} \
    For any $\varepsilon > 0,$ there exist $h_\ast > 0$ and $\delta_\ast > 0$ such that, for any initial value $x_0 \in B(R_0),$ any $h \in (0, h_\ast],$ any partition $\Delta$ with the diameter $\diam(\Delta) \leq \delta_\ast$ and any realizations $v(\cdot) \in \mathcal{V},$ $\widetilde{u}(\cdot) \in \mathcal{U},$ if realizations $u(\cdot) \in \mathcal{U},$ $\widetilde{v}(\cdot) \in \mathcal{V}$ are formed according to mutual aiming procedure $(\ref{procedure}),$ then the corresponding motions $x(\cdot) = x(\cdot; x_0; u(\cdot), v(\cdot))$ of system $(\ref{system_u_v}),$ $(\ref{initial_condition_u_v})$ and $y(\cdot) = y(\cdot; x_0; h; \widetilde{u}(\cdot), \widetilde{v}(\cdot))$ of approximating system $(\ref{system_y_u_v}),$ $(\ref{initial_condition_y_u_v})$ satisfy inequality $(\ref{thm_approximation_main}).$
\end{theorem}
\begin{proof} \
    According to Propositions~\ref{prop_RL_derivative_Lipshitz} and~\ref{prop_GL_derivative_Lipshitz}, by the number $\overline{L} > 0$ from Proposition~\ref{Prop_motion_y_u_v_properties}, let us choose $M > 0$ and $\overline{K} > 0,$ $\overline{M} > 0$ such that, for any $h > 0$ and any $y(\cdot) \in \Lip_{\overline{L}}^0,$ the inequalities (\ref{thm_a_p_M}) and
    \begin{equation} \label{thm_a_c_K}
            \|h^{1 - \alpha} (\Delta_h^{1 - \alpha} y)(t) - h^{1 - \alpha} (\Delta_h^{1 - \alpha} y)(\tau) \|
            \leq \overline{K} |t - \tau|^\alpha, \quad t, \tau \in [0, T],
    \end{equation}
    are valid. Let $\overline{R} > 0,$ $\overline{H} > 0$ be taken from Proposition~\ref{Prop_motion_u_v_properties} and $\overline{R}_1 = \overline{R} + R_0 + M + \overline{M},$ $\overline{H}_1 = \overline{H} + \overline{K}.$ By the number $\overline{R}_1,$ let us choose $\lambda_g > 0$ according to ($g.2$).
    Let $\varepsilon > 0$ be fixed. Let $\eta > 0$ satisfy the first inequality in (\ref{thm_a_p_eta_zeta}) where we substitute $\lambda_g$ instead of $\lambda_f,$ and $\zeta > 0$ be such that
    \begin{equation*}
        \zeta \leq \min \big\{ \eta / (32 (1 + \overline{R}_1) c_g), \eta/(32 \lambda_g \overline{R}_1) \big\},
    \end{equation*}
    where $c_g$ is the constant from $(g.3).$ Let $h_\ast > 0$ be chosen by Lemma~\ref{lem_RL_GL_derivatives} such that, for any $h \in (0, h_\ast]$ and any $y(\cdot) \in \Lip^0_{\overline{L}},$ inequality (\ref{thm_a_p_h_ast}) holds. Let $\delta_\ast > 0$ be such that $\delta_\ast^\alpha \leq \zeta/\overline{H}_1$ and, due to ($g.1$), for any $t, \tau \in [0, T],$ $x, y \in B(\overline{R}_1),$ $u \in P,$ $v \in Q,$ if $|t - \tau| \leq \delta_\ast$ and $\|x - y\| \leq \overline{H}_1 \delta_\ast^\alpha,$ then
    \begin{equation*}
        \|g(t, x, u, v) - g(\tau, y, u, v)\| \leq \eta / (16 \overline{R}_1).
    \end{equation*}
    Let us show that these $h_\ast$ and $\delta_\ast$ satisfy the statement of the theorem.

    Let $x_0 \in B(R_0,)$ $h \in (0, h_\ast]$ and $\Delta$ be a partition with the diameter $\delta = \diam(\Delta) \leq \delta_\ast.$ Let $v(\cdot) \in \mathcal{V},$ $\widetilde{u}(\cdot) \in \mathcal{U},$ and realizations $u(\cdot) \in \mathcal{U},$ $\widetilde{v}(\cdot) \in \mathcal{V}$ be formed according to (\ref{procedure}). Let $x(\cdot) = x(\cdot; x_0; u(\cdot), v(\cdot))$ and $y(\cdot) = y(\cdot; x_0; h; \widetilde{u}(\cdot), \widetilde{v}(\cdot))$ be, respectively, the motions of systems (\ref{system_u_v}), (\ref{initial_condition_u_v}) and (\ref{system_y_u_v}), (\ref{initial_condition_y_u_v}). Note that, by the choice of $\overline{L},$ the inclusion $y(\cdot) \in \Lip_{\overline{L}}^0$ holds, and therefore, relations (\ref{thm_a_p_M}), (\ref{thm_a_p_h_ast}) and (\ref{thm_a_c_K}) are valid. Let us consider the functions $s(\cdot),$ $\widetilde{x}(\cdot)$ and $V(\cdot)$ defined by (\ref{thm_a_p_s})--(\ref{thm_a_p_V}). In accordance with the proof of Theorem~\ref{thm_approximation}, in order to prove inequality (\ref{thm_approximation_main}), by the choice of $\eta$ and $h_\ast,$ it is sufficient to show that inequality (\ref{thm_a_p_DV_estimate_final}) is valid where we substitute $\lambda_g$ instead of $\lambda_f,$ i.e.,
    \begin{equation} \label{thm_c_p_estimate_DV_final}
        (D^\alpha V)(t) \leq 2 \lambda_g V(t) + \eta \text{ for a.e. } t \in [0, T].
    \end{equation}

    By analogy with (\ref{thm_a_p_DV_estimate}), for almost every $t \in [0, T],$ we have
    \begin{equation} \label{thm_c_p_estimate}
        (D^\alpha V)(t) \leq 2 \langle s(t), g(t, x(t), u(t), v(t)) \rangle
        - 2 \langle s(t), g(t, \widetilde{x}(t), \widetilde{u}(t), \widetilde{v}(t))\rangle.
    \end{equation}
    Let us fix $j \in \overline{1, k}$ and estimate the right-hand side of this inequality for $t \in [\tau_j, \tau_{j+1}).$ Note that, by the choice of $\overline{R}_1,$
    \begin{equation*}
        \max \big\{\|x(t)\|, \|x(\tau_j)\|, \|\widetilde{x}(t)\|, \|\widetilde{x}(\tau_j)\|, \|s(t)\|, \|\widetilde{s}_j\| \big\} \leq \overline{R}_1,
    \end{equation*}
    where $\widetilde{s}_j$ is defined according to (\ref{procedure}), and, by the choice of $\overline{H}_1,$ $\delta_\ast$ and $h_\ast,$
    \begin{equation*}
        \begin{array}{c}
            \|x(t) - x(\tau_j)\| \leq \overline{H} \delta^\alpha \leq \overline{H}_1 \delta^\alpha, \quad
            \|\widetilde{x}(t) - \widetilde{x}(\tau_j)\| \leq \overline{K} \delta^\alpha \leq \overline{H}_1 \delta^\alpha, \\[0.5em]
            \|s(t) - \widetilde{s}_j\|
            = \|x(t) - (D^{1 - \alpha} y)(t) - x(\tau_j) + h^{\alpha - 1} (\Delta_h^{1 - \alpha} y) (\tau_j)\| \\[0.5em]
            \leq \|x(t) - x(\tau_j)\|
            + \|(D^{1 - \alpha} y)(t) - h^{\alpha - 1} (\Delta_h^{1 - \alpha} y)(t)\| 
            + \|h^{\alpha - 1} (\Delta_h^{1 - \alpha} y)(t) - h^{\alpha - 1} (\Delta_h^{1 - \alpha} y) (\tau_j)\|
            \leq \overline{H}_1 \delta^\alpha + \zeta \leq 2 \zeta.
        \end{array}
    \end{equation*}
    Estimating the first term in (\ref{thm_c_p_estimate}), by $(g.3)$ and the choice of $\zeta,$ we derive
    \begin{equation} \label{thm_c_p_estimate_fst_1}
        \begin{array}{c}
            \langle s(t), g(t, x(t), u(t), v(t)) \rangle
            \leq \langle \widetilde{s}_j, g(t, x(t), u(t), v(t)) \rangle
            + \|s(t) - \widetilde{s}_j\| \|g(t, x(t), u(t), v(t))\| \\[0.5em]
            \leq \langle \widetilde{s}_j, g(t, x(t), u(t), v(t)) \rangle
            + 2 (1 + \overline{R}_1) c_g \zeta
            \leq \langle \widetilde{s}_j, g(t, x(t), u(t), v(t)) \rangle + \eta/16,
        \end{array}
    \end{equation}
    Further, due to the choice of $\delta_\ast,$ we obtain
    \begin{equation} \label{thm_c_p_estimate_fst_2}
        \begin{array}{c}
            \langle \widetilde{s}_j, g(t, x(t), u(t), v(t)) \rangle
            \leq \langle \widetilde{s}_j, g(\tau_j, x(\tau_j), u(t), v(t)) \rangle
            + \|\widetilde{s}_j\| \|g(t, x(t), u(t), v(t)) - g(\tau_j, x(\tau_j), u(t), v(t)) \| \\[0.5em]
            \leq \langle \widetilde{s}_j, g(\tau_j, x(\tau_j), u(t), v(t)) \rangle + \eta/16.
        \end{array}
    \end{equation}
    Let us consider the Hamiltonian of system (\ref{system_u_v}) defined by
    \begin{equation*}
        H(t, x, s) = \min_{u \in P} \max_{v \in Q} \langle s, g(t, x, u, v) \rangle, \quad t \in [0, T], \quad x, s \in \mathbb{R}^n.
    \end{equation*}
    Then, according to the choice of $u(t) = u_j$ in (\ref{procedure}), we have
    \begin{equation} \label{thm_c_p_estimate_fst_3}
        \langle \widetilde{s}_j, g(\tau_j, x(\tau_j), u(t), v(t)) \rangle
        \leq \max_{v \in Q} \langle \widetilde{s}_j, g(\tau_j, x(\tau_j), u_j, v) \rangle
        = H(\tau_j, x(\tau_j), \widetilde{s}_j).
    \end{equation}
    Therefore, from (\ref{thm_c_p_estimate_fst_1})--(\ref{thm_c_p_estimate_fst_3}) we deduce
    \begin{equation} \label{thm_c_p_estimate_fst_main}
        \langle s(t), g(t, x(t), u(t), v(t)) \rangle \leq H(\tau_j, x(\tau_j), \widetilde{s}_j) + \eta/8.
    \end{equation}
    Due to $(g.4),$ for the second term in (\ref{thm_c_p_estimate}), one can similarly obtain
    \begin{equation} \label{thm_c_p_estimate_snd_main}
        \langle s(t), g(t, \widetilde{x}(t), u(t), v(t)) \rangle
        \geq H(\tau_j, \widetilde{x}(\tau_j), \widetilde{s}_j) - \eta/8.
    \end{equation}
    Further, by the choice of $\lambda_g,$ we have
    \begin{equation} \label{thm_c_p_estimate_H-H}
        H(\tau_j, x(\tau_j), \widetilde{s}_j) - H(\tau_j, \widetilde{x}(\tau_j), \widetilde{s}_j)
        \leq \lambda_g \|x(\tau_j) - \widetilde{x}(\tau_j)\| \|\widetilde{s}_j\|
        = \lambda_g \|\widetilde{s}_j\|^2.
    \end{equation}
    Note that, by the choice of $\zeta,$ the following inequality is valid:
    \begin{equation} \label{thm_c_p_estimate_s_j}
        \|\widetilde{s}_j\|^2 \leq \|s(t)\|^2 + 2 (\|s(t) - \widetilde{s}_j\| + 2 \|s(t)\|) \|s(t) - \widetilde{s}_j\|
        \leq V(t) + 8 \overline{R}_1 \zeta \leq V(t) + \eta/(4 \lambda_g).
    \end{equation}
    From (\ref{thm_c_p_estimate}) and (\ref{thm_c_p_estimate_fst_main})--(\ref{thm_c_p_estimate_s_j}) we derive (\ref{thm_c_p_estimate_DV_final}). The theorem is proved.
\end{proof}

\section{\; Example~2}\label{sec_Example_2}

Let us illustrate Theorem~\ref{thm_closeness} by an example. Let a motion of the conflict-controlled dynamical system be described by the fractional differential equations
\begin{equation} \label{system_u_v_ex}
    \begin{array}{c}
        \begin{cases}
            (^C D^{0.7} x_1) (t) = x_2(t), \\[0.5em]
            (^C D^{0.7} x_2) (t) = - 0.5 \, \sin(x_1(t)) + 0.5 \, u(t) + 0.5 \, v(t),
        \end{cases} \\[1.7em]
        t \in [0, 10], \quad x(t) = (x_1(t), x_2(t)) \in \mathbb{R}^2, \quad u(t) \in [-1, 1], \quad v(t) \in [-1, 1],
    \end{array}
\end{equation}
with the initial condition
\begin{equation}\label{initial_condition_u_v_ex}
    x(0) = x_0 = (0, 0.5).
\end{equation}
For system (\ref{system_u_v_ex}), (\ref{initial_condition_u_v_ex}), the corresponding approximating system (\ref{system_y_u_v}), (\ref{initial_condition_y_u_v}) was considered, and mutual aiming procedure (\ref{procedure}) between these systems was simulated. Namely, realizations $u(\cdot)$ in the initial system and $\widetilde{v}(\cdot)$ in the approximating system were formed according to (\ref{procedure}), while realizations $v(\cdot)$ and $\widetilde{u}(\cdot)$ were chosen as follows:
\begin{equation*}
    v(t) = \sin(3 \, t), \quad \widetilde{u}(t) = \cos(2 \, t), \quad t \in [0, 10].
\end{equation*}
As in Example~1, for the numerical construction of the motions, we use the forward Euler methods with the step $0.001.$

The following two cases were considered. In the first one, we choose $h = 0.1$ and the partition $\Delta$ of the segment $[0, 10]$ with the constant step $\delta = 0.02.$ The obtained difference between the realized motion $x(\cdot)$ of system (\ref{system_u_v_ex}), (\ref{initial_condition_u_v_ex}) and its approximation $\widetilde{x}(t) = x_0 + h^{-0.3} (\Delta_h^{0.3} y)(t),$ $t \in [0, 10],$ is
\begin{equation*}
    \max_{t \in [0, 10]} \|x(t) - \widetilde{x}(t)\| \approx 0.114.
\end{equation*}
In the second case, we choose $h = 0.01,$ $\delta = 0.005$ and obtain
\begin{equation*}
    \max_{t \in [0, 10]} \|x(t) - \widetilde{x}(t)\| \approx 0.046.
\end{equation*}
The simulations results are shown below in Fig.~6.1.

\vspace*{-2.5em}

\begin{center}
    \hspace*{-2em}
\begingroup
  \makeatletter
  \providecommand\color[2][]{%
    \GenericError{(gnuplot) \space\space\space\@spaces}{%
      Package color not loaded in conjunction with
      terminal option `colourtext'%
    }{See the gnuplot documentation for explanation.%
    }{Either use 'blacktext' in gnuplot or load the package
      color.sty in LaTeX.}%
    \renewcommand\color[2][]{}%
  }%
  \providecommand\includegraphics[2][]{%
    \GenericError{(gnuplot) \space\space\space\@spaces}{%
      Package graphicx or graphics not loaded%
    }{See the gnuplot documentation for explanation.%
    }{The gnuplot epslatex terminal needs graphicx.sty or graphics.sty.}%
    \renewcommand\includegraphics[2][]{}%
  }%
  \providecommand\rotatebox[2]{#2}%
  \@ifundefined{ifGPcolor}{%
    \newif\ifGPcolor
    \GPcolorfalse
  }{}%
  \@ifundefined{ifGPblacktext}{%
    \newif\ifGPblacktext
    \GPblacktexttrue
  }{}%
  \let\gplgaddtomacro\g@addto@macro
  \gdef\gplbacktext{}%
  \gdef\gplfronttext{}%
  \makeatother
  \ifGPblacktext
    \def\colorrgb#1{}%
    \def\colorgray#1{}%
  \else
    \ifGPcolor
      \def\colorrgb#1{\color[rgb]{#1}}%
      \def\colorgray#1{\color[gray]{#1}}%
      \expandafter\def\csname LTw\endcsname{\color{white}}%
      \expandafter\def\csname LTb\endcsname{\color{black}}%
      \expandafter\def\csname LTa\endcsname{\color{black}}%
      \expandafter\def\csname LT0\endcsname{\color[rgb]{1,0,0}}%
      \expandafter\def\csname LT1\endcsname{\color[rgb]{0,1,0}}%
      \expandafter\def\csname LT2\endcsname{\color[rgb]{0,0,1}}%
      \expandafter\def\csname LT3\endcsname{\color[rgb]{1,0,1}}%
      \expandafter\def\csname LT4\endcsname{\color[rgb]{0,1,1}}%
      \expandafter\def\csname LT5\endcsname{\color[rgb]{1,1,0}}%
      \expandafter\def\csname LT6\endcsname{\color[rgb]{0,0,0}}%
      \expandafter\def\csname LT7\endcsname{\color[rgb]{1,0.3,0}}%
      \expandafter\def\csname LT8\endcsname{\color[rgb]{0.5,0.5,0.5}}%
    \else
      \def\colorrgb#1{\color{black}}%
      \def\colorgray#1{\color[gray]{#1}}%
      \expandafter\def\csname LTw\endcsname{\color{white}}%
      \expandafter\def\csname LTb\endcsname{\color{black}}%
      \expandafter\def\csname LTa\endcsname{\color{black}}%
      \expandafter\def\csname LT0\endcsname{\color{black}}%
      \expandafter\def\csname LT1\endcsname{\color{black}}%
      \expandafter\def\csname LT2\endcsname{\color{black}}%
      \expandafter\def\csname LT3\endcsname{\color{black}}%
      \expandafter\def\csname LT4\endcsname{\color{black}}%
      \expandafter\def\csname LT5\endcsname{\color{black}}%
      \expandafter\def\csname LT6\endcsname{\color{black}}%
      \expandafter\def\csname LT7\endcsname{\color{black}}%
      \expandafter\def\csname LT8\endcsname{\color{black}}%
    \fi
  \fi
    \setlength{\unitlength}{0.0500bp}%
    \ifx\gptboxheight\undefined%
      \newlength{\gptboxheight}%
      \newlength{\gptboxwidth}%
      \newsavebox{\gptboxtext}%
    \fi%
    \setlength{\fboxrule}{0.5pt}%
    \setlength{\fboxsep}{1pt}%
\begin{picture}(4896.00,5328.00)%
    \gplgaddtomacro\gplbacktext{%
      \csname LTb\endcsname%
      \put(726,1295){\makebox(0,0)[r]{\strut{}$-1$}}%
      \csname LTb\endcsname%
      \put(726,2023){\makebox(0,0)[r]{\strut{}$-0.5$}}%
      \csname LTb\endcsname%
      \put(726,2752){\makebox(0,0)[r]{\strut{}$0$}}%
      \csname LTb\endcsname%
      \put(726,3480){\makebox(0,0)[r]{\strut{}$0.5$}}%
      \csname LTb\endcsname%
      \put(726,4208){\makebox(0,0)[r]{\strut{}$x$}}%
      \csname LTb\endcsname%
      \put(858,1075){\makebox(0,0){\strut{}$0$}}%
      \csname LTb\endcsname%
      \put(1586,1075){\makebox(0,0){\strut{}$2$}}%
      \csname LTb\endcsname%
      \put(2314,1075){\makebox(0,0){\strut{}$4$}}%
      \csname LTb\endcsname%
      \put(3043,1075){\makebox(0,0){\strut{}$6$}}%
      \csname LTb\endcsname%
      \put(3771,1075){\makebox(0,0){\strut{}$8$}}%
      \csname LTb\endcsname%
      \put(4499,1075){\makebox(0,0){\strut{}$t$}}%
    }%
    \gplgaddtomacro\gplfronttext{%
      \csname LTb\endcsname%
      \put(3188,1840){\makebox(0,0)[r]{\strut{}$x_1$}}%
      \csname LTb\endcsname%
      \put(3188,1620){\makebox(0,0)[r]{\strut{}$x_2$}}%
      \csname LTb\endcsname%
      \put(4135,1840){\makebox(0,0)[r]{\strut{}$\widetilde{x}_1$}}%
      \csname LTb\endcsname%
      \put(4135,1620){\makebox(0,0)[r]{\strut{}$\widetilde{x}_2$}}%
    }%
    \gplbacktext
    \put(0,0){\includegraphics{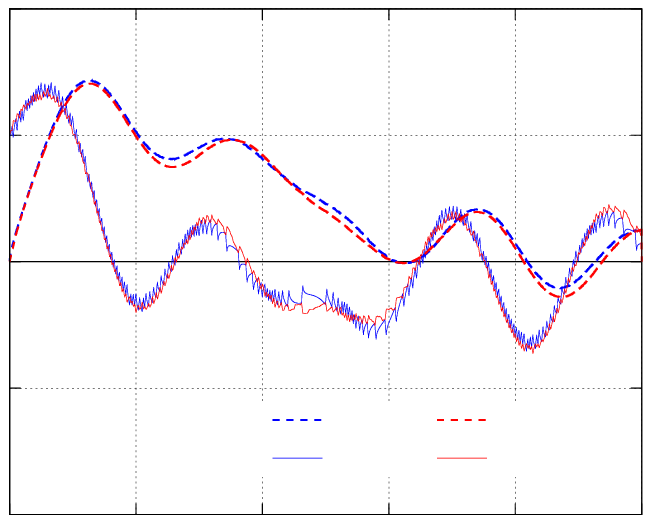}}%
    \gplfronttext
  \end{picture}%
\endgroup

    \hspace*{-1em}
\begingroup
  \makeatletter
  \providecommand\color[2][]{%
    \GenericError{(gnuplot) \space\space\space\@spaces}{%
      Package color not loaded in conjunction with
      terminal option `colourtext'%
    }{See the gnuplot documentation for explanation.%
    }{Either use 'blacktext' in gnuplot or load the package
      color.sty in LaTeX.}%
    \renewcommand\color[2][]{}%
  }%
  \providecommand\includegraphics[2][]{%
    \GenericError{(gnuplot) \space\space\space\@spaces}{%
      Package graphicx or graphics not loaded%
    }{See the gnuplot documentation for explanation.%
    }{The gnuplot epslatex terminal needs graphicx.sty or graphics.sty.}%
    \renewcommand\includegraphics[2][]{}%
  }%
  \providecommand\rotatebox[2]{#2}%
  \@ifundefined{ifGPcolor}{%
    \newif\ifGPcolor
    \GPcolorfalse
  }{}%
  \@ifundefined{ifGPblacktext}{%
    \newif\ifGPblacktext
    \GPblacktexttrue
  }{}%
  \let\gplgaddtomacro\g@addto@macro
  \gdef\gplbacktext{}%
  \gdef\gplfronttext{}%
  \makeatother
  \ifGPblacktext
    \def\colorrgb#1{}%
    \def\colorgray#1{}%
  \else
    \ifGPcolor
      \def\colorrgb#1{\color[rgb]{#1}}%
      \def\colorgray#1{\color[gray]{#1}}%
      \expandafter\def\csname LTw\endcsname{\color{white}}%
      \expandafter\def\csname LTb\endcsname{\color{black}}%
      \expandafter\def\csname LTa\endcsname{\color{black}}%
      \expandafter\def\csname LT0\endcsname{\color[rgb]{1,0,0}}%
      \expandafter\def\csname LT1\endcsname{\color[rgb]{0,1,0}}%
      \expandafter\def\csname LT2\endcsname{\color[rgb]{0,0,1}}%
      \expandafter\def\csname LT3\endcsname{\color[rgb]{1,0,1}}%
      \expandafter\def\csname LT4\endcsname{\color[rgb]{0,1,1}}%
      \expandafter\def\csname LT5\endcsname{\color[rgb]{1,1,0}}%
      \expandafter\def\csname LT6\endcsname{\color[rgb]{0,0,0}}%
      \expandafter\def\csname LT7\endcsname{\color[rgb]{1,0.3,0}}%
      \expandafter\def\csname LT8\endcsname{\color[rgb]{0.5,0.5,0.5}}%
    \else
      \def\colorrgb#1{\color{black}}%
      \def\colorgray#1{\color[gray]{#1}}%
      \expandafter\def\csname LTw\endcsname{\color{white}}%
      \expandafter\def\csname LTb\endcsname{\color{black}}%
      \expandafter\def\csname LTa\endcsname{\color{black}}%
      \expandafter\def\csname LT0\endcsname{\color{black}}%
      \expandafter\def\csname LT1\endcsname{\color{black}}%
      \expandafter\def\csname LT2\endcsname{\color{black}}%
      \expandafter\def\csname LT3\endcsname{\color{black}}%
      \expandafter\def\csname LT4\endcsname{\color{black}}%
      \expandafter\def\csname LT5\endcsname{\color{black}}%
      \expandafter\def\csname LT6\endcsname{\color{black}}%
      \expandafter\def\csname LT7\endcsname{\color{black}}%
      \expandafter\def\csname LT8\endcsname{\color{black}}%
    \fi
  \fi
    \setlength{\unitlength}{0.0500bp}%
    \ifx\gptboxheight\undefined%
      \newlength{\gptboxheight}%
      \newlength{\gptboxwidth}%
      \newsavebox{\gptboxtext}%
    \fi%
    \setlength{\fboxrule}{0.5pt}%
    \setlength{\fboxsep}{1pt}%
\begin{picture}(4896.00,5328.00)%
    \gplgaddtomacro\gplbacktext{%
      \csname LTb\endcsname%
      \put(726,1295){\makebox(0,0)[r]{\strut{}$-1$}}%
      \csname LTb\endcsname%
      \put(726,2023){\makebox(0,0)[r]{\strut{}$-0.5$}}%
      \csname LTb\endcsname%
      \put(726,2752){\makebox(0,0)[r]{\strut{}$0$}}%
      \csname LTb\endcsname%
      \put(726,3480){\makebox(0,0)[r]{\strut{}$0.5$}}%
      \csname LTb\endcsname%
      \put(726,4208){\makebox(0,0)[r]{\strut{}$x$}}%
      \csname LTb\endcsname%
      \put(858,1075){\makebox(0,0){\strut{}$0$}}%
      \csname LTb\endcsname%
      \put(1586,1075){\makebox(0,0){\strut{}$2$}}%
      \csname LTb\endcsname%
      \put(2314,1075){\makebox(0,0){\strut{}$4$}}%
      \csname LTb\endcsname%
      \put(3043,1075){\makebox(0,0){\strut{}$6$}}%
      \csname LTb\endcsname%
      \put(3771,1075){\makebox(0,0){\strut{}$8$}}%
      \csname LTb\endcsname%
      \put(4499,1075){\makebox(0,0){\strut{}$t$}}%
    }%
    \gplgaddtomacro\gplfronttext{%
      \csname LTb\endcsname%
      \put(3188,1840){\makebox(0,0)[r]{\strut{}$x_1$}}%
      \csname LTb\endcsname%
      \put(3188,1620){\makebox(0,0)[r]{\strut{}$x_2$}}%
      \csname LTb\endcsname%
      \put(4135,1840){\makebox(0,0)[r]{\strut{}$\widetilde{x}_1$}}%
      \csname LTb\endcsname%
      \put(4135,1620){\makebox(0,0)[r]{\strut{}$\widetilde{x}_2$}}%
    }%
    \gplbacktext
    \put(0,0){\includegraphics{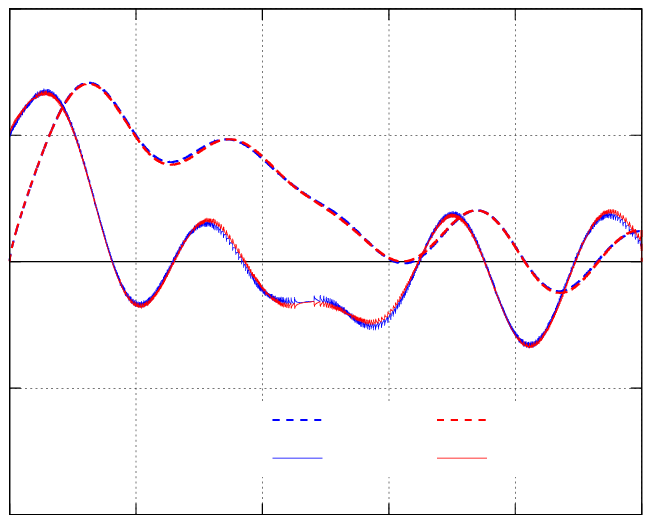}}%
    \gplfronttext
  \end{picture}%
\endgroup

    \hspace*{-2em}

    \vspace*{-1.8em}

    Fig. 6.1: The realized motions $x(\cdot)$ of system (\ref{system_u_v_ex}), (\ref{initial_condition_u_v_ex}) and their approximations $\widetilde{x}(\cdot)$ for $h = 0.1,$ $\delta = 0.02$ and $h = 0.01,$ $\delta = 0.005.$
\end{center}

\section{\; Conclusion} \label{sec_Conclusion}

In the paper, a conflict-controlled dynamical system described by a fractional differential equation with the Caputo derivative of an order $\alpha \in (0, 1)$ is considered. An approximation of this system by a system described by a functional-differential equation of a retarded type with the usual first order derivative is proposed. In order to ensure the desired proximity between the initial and the approximating systems, a mutual aiming procedure is elaborated. In the further applications, the obtained results can be used for reducing different control problems in fractional order systems, including control problems under disturbances and differential games, to control problems in functional-differential systems (see, e.g., \cite{YuSOsipov_1971,NYuLukoyanov_2001,NYuLukoyanov_2003}). Namely, one can consider control schemes that use the approximating system as a modelling guide (see, e.g., \cite[\S~8.2]{NNKrasovskii_AISubbotin_1988} and also \cite{NNKrasovskii_ANKotelnikova_2012,NYuLukoyanov_ARPlaksin_2015,NYuLukoyanov_ARPlaksin_2016}). In this case, control actions $u(t)$ in the initial system and $\widetilde{v}(t)$ in the approximating system are chosen according to the mutual aiming procedure, and the desired quality of a control process is attained due to the choice of control actions $\widetilde{u}(t)$ in the approximating system. In particular, such an approach allows, via the proposed approximations, to develop theory and numerical methods for solving different control problems in fractional order systems.

\emergencystretch=\hsize

\begin{center}
\rule{6 cm}{0.02 cm}
\end{center}

\end{document}


\vskip 2mm

\begin{thebibliography}{99}

\bibitem{SGSamko_AAKilbas_OIMarichev_1993}
 S. G. Samko, A. A. Kilbas, O. I. Marichev,
 \textit{Fractional Integrals and Derivatives. Theory and Applications},
 Gordon \& Breach Sci. Publishers, 1993.

\bibitem{IPodlubny_1999}
 I. Podlubny,
 \textit{Fractional Differential Equations},
 Academic Press, New York, 1999.

\bibitem{AAKilbas_HMSrivastava_JJTrujillo_2006}
 A. A. Kilbas, H. M. Srivastava, J. J. Trujillo,
 \textit{Theory and Applications of Fractional Differential Equations},
 Elsevier, 2006.

\bibitem{KDiethelm_2010}
 K. Diethelm,
 \textit{The Analysis of Fractional Differential Equations},
 Springer-Verlag, Berlin, Heidelberg, 2010.

\bibitem{CLi_FZeng_2015}
 C. Li, F. Zeng,
 \textit{Numerical Methods for Fractional Calculus},
 Chapman and Hall, New York, 2015.

\bibitem{JKHale_SMVLunel_1993}
 J. K. Hale, S. M. V. Lunel,
 \textit{Introduction to Functional Differential Equations},
 Springer, New York, 1993.

\bibitem{MIGomoyunov_2017}
 M. I. Gomoyunov,
 Fractional derivatives of convex Lyapunov functions and control problems in fractional order systems,
 \textit{arXiv:1710.07003} (2017). (submitted to \textit{Frac. Calc. Appl. Anal.})

\bibitem{NNKrasovskii_ANKotelnikova_2012}
 N. N. Krasovskii, A. N. Kotelnikova,
 Stochastic guide for a time-delay object in a positional differential game,
 \textit{Proc. Steklov Inst. Math.} \textbf{277}, suppl. 1, 145--151 (2012).

\bibitem{NYuLukoyanov_ARPlaksin_2015}
 N. Yu. Lukoyanov, A. R. Plaksin,
 On approximations of time-delay control systems,
 \textit{IFAC-PapersOnLine}, \textbf{28} (25), 178--182 (2015).

\bibitem{NYuLukoyanov_ARPlaksin_2016}
 N. Yu. Lukoyanov, A. R. Plaksin,
 On the approximation of nonlinear conflict-controlled systems of neutral type,
 \textit{Proc. Steklov Inst. Math.} \textbf{292}, suppl. 1, 182--195 (2016).

\bibitem{NNKrasovskii_AISubbotin_1988}
 N. N. Krasovskii, A. I. Subbotin,
 \textit{Game-Theoretical Control Problems},
 Springer-Verlag, New York, 1988.

\bibitem{YuSOsipov_1971}
 Yu. S. Osipov,
 Differential games of systems with aftereffect,
 \textit{Dokl. Akad. Nauk SSSR}, \textbf{196} (4), 779--782 (1971).

\bibitem{NYuLukoyanov_2001}
 N. Yu. Lukoyanov,
 The properties of the value functional of a differential game with hereditary information,
 \textit{J. Appl. Math. Mech.} \textbf{65} (3), 361--370 (2001).

\bibitem{NYuLukoyanov_2003}
 N. Yu. Lukoyanov,
 Functional Hamilton-Jacobi type equations with ci-derivatives in control problems with hereditary information,
 \textit{Nonlinear Funct. Anal. Appl.} \textbf{8} (4), 535--555 (2003).

\bibitem{UWestphal_1974}
 U. Westphal,
 An approach to fractional powers of operators via fractional differences,
 \textit{Proc. London Math. Soc.} \textbf{29} (3), 557--576 (1974).

\bibitem{RADeVore_GGLorentz_1993}
 R. A. DeVore, G. G. Lorentz,
 \textit{Constructive Approximation},
 Springer, Berlin, 1993.

\bibitem{LVKantorovich_GPAkilov_1982}
 L. V. Kantorovich, G. P. Akilov,
 \textit{Functional Analisys},
 Pergamon Press, Oxford, 1982.

\bibitem{AMZverkin_GAKemenskii_SBNorkin_LEElsgolts_1962}
 A. M. Zverkin, G. A. Kemenskii, S. B. Norkin and L. E. El'sgol'ts,
 Differential equations with a perturbed argument,
 \textit{Russ. Math. Surv.} \textbf{17} (2), 61--146 (1962).

\bibitem{AVKim_2015}
 A. V. Kim,
 \textit{i--Smooth Analysis: Theory and Applications},
 Wiley, 2015.
\end{thebibliography}
\end{document}